\documentclass{article}     

\usepackage{amsfonts}
\usepackage{amsmath}

\newtheorem{theorem}{Theorem}
\newtheorem{acknowledgements}[theorem]{Acknowledgements}

\newtheorem{corollary}{Corollary}

\newtheorem{definition}{Definition}

\newtheorem{lemma}{Lemma}

\newtheorem{proposition}{Proposition}
\newtheorem{remark}{Remark}

\newenvironment{proof}[1][Proof]{\noindent\textbf{#1.} }{\ \rule{0.5em}{0.5em}}

\begin{document}

\title{Number of Distinct Sites Visited by a Random Walk with Internal States
}


\author{P\'{e}ter N\'{a}ndori         
	}
\date{\vspace{-5ex}}



\maketitle

\begin{abstract}
In the classical paper of Dvoretzky-Erd\H{o}s \cite{Dv-Er}, asymptotics for
the expected value and the variance of the number of distinct sites visited
by a Simple Symmetric Random Walk were calculated. Here, these results are
generalized for Random Walks with Internal States. Moreover, both weak and
strong laws of large numbers are proved. As a tool for these results, the
error term of the local limit theorem in \cite{Kramli-Szasz}\ is also 
estimated.
\end{abstract}

\section{Introduction}
\label{Secintro}

The model of a random walk with internal states (or, alternatively, random
walk with internal degrees of freedom; briefly RWwIS) was introduced by
Sinai in 1981 in his Kyoto talk \cite{Sinai}. His aim was to get an
efficient tool for examining the Lorentz process (in this context, internal
states would represent the elements of the Markov partition or of a 
Markov sieve).  For this kind
of argument see, for instance, \cite{PSz}. Beside the Lorentz process, 
however, several 
other motivations and applications have appeared, among others, 
in some models of queueing systems, cf. \cite{Hughes} as for an extensive 
treatment of other motivations.
Nevertheless, the
investigation of this model is important for its own sake, as it is a 
manifest generalization of a gem of probability theory: the 
simple symmetric random walk. Let us begin with the definition of RWwIS
with the notation in \cite{Kramli-Szasz} and \cite{Kramli-Szasz2} (or of \cite{Kramli-Simanyi-Szasz}, 
where RWwIS served as a model of Fourier law of heat conduction).
 
\begin{definition}
Let E be a finite set. On the set $H=%
\mathbb{Z}
^{d}\times E$ $(d=1,2,...)$, the Markov chain $\xi _{n}=(\eta
_{n},\varepsilon _{n})$ is a random walk with internal states (RWwIS), if
for $\forall x_{n},x_{n+1}\in 
\mathbb{Z}
^{d},\quad j_{n},j_{n+1}\in E$
\end{definition}

\begin{equation*}
P(\xi _{n+1}=(x_{n+1},j_{n+1})|\xi
_{n}=(x_{n},j_{n}))=p_{x_{n+1}-x_{n},j_{n},j_{n+1}}.
\end{equation*}

In fact, $E$ could be countable, as well, but we will consider only the
finite case. We will denote $s=\#E$.

There are some basic assumptions which will throughout be supposed. These
are the following:

\begin{enumerate}
\item[(i)] $\left( \varepsilon _{0},\varepsilon _{1},...\right) $ -
obviously a Markov chain - is irreducible and aperiodic (its stationary
distribution will be denoted by $\mu $)

\item[(ii)] the arithmetics are trivial, with the notation in \cite%
{Kramli-Szasz}, $L=%
\mathbb{Z}
^{d}$

\item[(iii)] the expectation of one step is zero provided that $\varepsilon
_{0}$ is distributed according to its unique stationary measure

\item[(iv)] the covariance matrix, which is exactly defined in Section \ref%
{Sectlimittheo}, exists and is nonsingular.
\end{enumerate}

In general, we will assume that $\eta _{0}=0$.
Let $L_{d}\left( n\right) $ denote the number of distinct sites visited by $\left( \eta _k \right)_k$ up to $n$ steps. The expectation of $L_{d}\left( n\right) $ is $%
E_{d}\left( n\right) $, and the variance is $V_{d}\left( n\right) $. 
$\{ e_j \}_{j=1,...,s}$ is the standard basis in $\mathbb{R}^s$,
and  $\underline{1}=\left( 1,1,...,1\right) ^{T}$. Our aim
is to find asymptotics of $E_{d}\left( n\right) $, further, by using bounds
on $V_{d}\left( n\right) $, we want to prove weak and strong laws of large
numbers. Similar results in terms of simple symmetric random walks (which
will later on be referred to as SSRW) are found in \cite{Dv-Er}. Recently,
in the case of two dimensional Lorenz process, P\`{e}ne discussed the same
question in \cite{Pene}. There are numerous fairly new papers on $%
L_{d}\left( n\right) $ for random walks with independent steps (see \cite%
{Bass}\ and references wherein).

This paper is organized as follows: in Section \ref{Sectlimittheo} the main
theorem of \cite{Kramli-Szasz} is generalized. Namely, a remainder term of
the local limit theorem is computed, as it will be necessary for estimating $%
E_{2}\left( n\right) $. A further refinement of the local limit theorem will
also be given as it will be useful when proving the strong law of large
numbers in the plane. Although these results are used in the forthcoming
Sections, they can be interesting in their own rights. In Section \ref%
{Secthighdim}, the number of visited points in the high dimensional case,
i.e. when $d\geq 3$, is dealt with. We prove asymptotics for $E_{d}\left(
n\right) $, and estimate $V_{d}\left( n\right) $, from which we can prove
both the weak and strong laws of large numbers. In this Section, we will not
use the result of Section \ref{Sectlimittheo}, Theorem 5.2. in \cite%
{Kramli-Szasz}\ will be enough for our purposes. In Section \ref{Secttwodim}
the $d=2$ case is discussed. For $E_{2}\left( n\right) $, same asymptotics ($%
const\frac{n}{\log n}$) is found as in \cite{Dv-Er}, but with some different
constant. $V_{2}\left( n\right) $ is also estimated, and the weak law of
large numbers is also proved. The proof of the strong law in the plane is a
little bit cumbersome calculation, so it is postponed to Section \ref%
{Sectstronglaw}. In Section \ref{Sectonedim}, the one dimensional settings
are considered. This case requires a little bit different approach from the
previous ones (and is not treated in \cite{Dv-Er}), so the application of a
Tauberian theorem will be very useful. Section \ref{Sectfinalremarks} gives
some remarks.

\section{Preliminaries}
\label{Sectlimittheo}

\subsection{Local limit theorem with remainder term}
\label{subloc}

In this subsection, we calculate a remainder term for Theorem 5.2. in \cite%
{Kramli-Szasz}. Furthermore, another refinement of this theorem will be
proved, as it will be used when proving the strong law of large numbers in
the plane. First, we reformulate the mentioned theorem. We have to start
with some definitions. Denote%
\begin{eqnarray*}
A_{y} &=&\left( p_{y,j,k}\right) _{j,k=1,...,s}:%
\mathbb{C}
^{s}\rightarrow 
\mathbb{C}
^{s}, \\
Q &=&\underset{y\in 
\mathbb{Z}
^{d}}{\sum }A_{y}, \\
M_{l} &=&\underset{y\in 
\mathbb{Z}
^{d}}{\sum }y_{l}A_{y}, \\
\Sigma _{l,m} &=&\underset{y\in 
\mathbb{Z}
^{d}}{\sum }y_{l}y_{m}A_{y}.
\end{eqnarray*}%
So, the transition matrix of the Markov chain $\left( \varepsilon
_{0},\varepsilon _{1},...\right) $ is $Q$ and its unique stationary measure
is $\mu $.

\begin{theorem}
(Kr\'{a}mli-Sz\'{a}sz \cite{Kramli-Szasz}) Consider a RWwIS\ in $%
\mathbb{Z}
^{d}$ and assume that the matrix $\sigma =\left( \sigma _{l,m}\right)
_{1\leq l,m\leq d}$ whose elements are%
\begin{equation*}
\sigma _{l,m}=\left\langle \mu ,\Sigma _{l,m}\underline{1}\right\rangle -\left\langle
\mu ,M_{l}\left( Q-1\right) ^{-1}M_{m}\underline{1}\right\rangle -\left\langle \mu
,M_{m}\left( Q-1\right) ^{-1}M_{l}\underline{1}\right\rangle
\end{equation*}%
(which can be called a covariance matrix) is positive definite, then%
\begin{equation*}
\underset{\left( x,k\right) \in H}{\sum }\left\vert P\left( \xi
_{n}=(x,k)|\xi _{0}=(0,j)\right) -n^{-d/2}\mu _{k}g_{\sigma }\left( \frac{x}{%
\sqrt{n}}\right) \right\vert \rightarrow 0
\end{equation*}%
as $n\rightarrow \infty $, where $g_{\sigma }\left( x\right) $ denotes the
density of a Gaussian distribution with mean $0$ and covariance matrix $%
\sigma $.
\end{theorem}

Of course, the condition concerning the positive definiteness of the matrix
in one dimension means $\sigma >0$. We omit the proof, it can be found in 
\cite{Dv-Er}. In fact, there is a typo in \cite{Dv-Er} as they write $%
n^{-1/2}$ instead of $n^{-d/2}$ but it is easy to correct it even in the
proof.

Our calculation will be similar to the one of \cite{Kramli-Szasz}. The main
point is that while in \cite{Kramli-Szasz} it is sufficient to consider the
Taylor expansion of the largest eigenvalue up to the quadratic term, now, we
have to calculate the third term, as well.

Define the Fourier transform 
\begin{equation*}
\alpha (t)=\underset{y\in 
\mathbb{Z}
^{d}}{\sum }\exp \left( i\left\langle t,y\right\rangle \right) A_{y},t\in %
\left[ -\pi ,\pi \right] ^{d}.
\end{equation*}%
Now, we have to consider the Taylor expansion of the largest eigenvalue of $%
\alpha (t)$, which is denoted by $\lambda (t)$, up to the third term.

Let us first assume that $d=1$. From our basic assumptions it follows that $%
M=\underset{y\in 
\mathbb{Z}
}{\sum }yA_{y}$ and $\Sigma =\underset{y\in 
\mathbb{Z}
}{\sum }y^{2}A_{y}$ are convergent series. But now, we also suppose the
absolute convergence of 
\begin{equation}
\Xi =\underset{y\in 
\mathbb{Z}
}{\sum }y^{3}A_{y}.  \label{harmadikmomentum}
\end{equation}%
The existence of $M,\Sigma $ and $\Xi $ implies 
\begin{equation}
\alpha (t)=Q+itM-\frac{t^{2}}{2}\Sigma -\frac{it^{3}}{6}\Xi +o(t^{3})\quad
(t\rightarrow 0).  \label{alfat}
\end{equation}%
Now, by perturbation theoretic means (i.e. the straightforward extension of
Theorem 5.11. of Chapter II. in \cite{Kato}) it can be easily proved that%
\begin{equation}
\lambda (t)=1+r_{1}t+\frac{r_{2}}{2}t^{2}+\frac{r_{3}}{6}t^{3}+o(t^{3})\quad
(t\rightarrow 0).  \label{largesteigenvalue}
\end{equation}%
From \cite{Kramli-Szasz} we know that $r_{1}=0$ and $r_{2}=-\left\langle
\Sigma \underline{1},\mu \right\rangle 
+2\left\langle M(Q-1)^{-1}M,\mu \right\rangle $.

Using the notation $\sigma ^{2}=-r_{2}$ we can now formulate our theorem:

\begin{theorem}
\label{limittheorem}For a one dimensional RWwIS the existence of (\ref%
{harmadikmomentum}) imply 
\begin{eqnarray*}
P(\xi _{n} &=&(x,k)|\xi _{0}=(0,j))- \\ 
&-& \mu _{k}\frac{1}{\sqrt{2\pi n}\sigma }%
\exp \left( -\frac{x^{2}}{2n\sigma ^{2}}\right) \left[ 1-\frac{ir_{3}}{6}%
x\left( 3\sigma ^{2}n-x^{2}\right) \frac{1}{\sigma ^{6}}\frac{1}{n^{2}}%
\right] = o\left( \frac{1}{n}\right) ,
\end{eqnarray*}%
where the small order is uniform in $x$.
\end{theorem}

\begin{proof}
The proof is similar to the one of Theorem 2.1. in \cite{Kramli-Szasz}.
In the neighborhood of the origin, we have $\alpha ^n (t) = \lambda ^n (t) p(t)
+ b_n (t)$, where $p$ is the projector to the eigenspace associated to 
$\lambda (t)$, and $b_n (t)$ is the contribution of the other eigenvalues.
The term $b_n (t)$ is in $O ( \alpha ^n)$ for some $\alpha \in (0,1)$.

Because of (\ref{largesteigenvalue}) we have%
\begin{equation}
\alpha ^{n}\left( t\right) = 
\left( \underline{1} \mu ^{T} + t p'(0) + O(t^2) \right)
\left( 1-\frac{\sigma ^{2}t^{2}}{2}+%
\frac{r_{3}}{6}t^{3}+o\left( t^{3}\right) \right) ^{n} + b_n (t) . 
\label{alfant}
\end{equation}%
Elementary calculations show that%
\begin{equation}
\left( 1-\frac{\sigma ^{2}s^{2}}{2n}+\frac{r_{3}}{6}\frac{s^{3}}{n^{\frac{3}{%
2}}}+o\left( \frac{s^{3}}{n^{\frac{3}{2}}}\right) \right) ^{n}=\exp \left( -%
\frac{\sigma ^{2}s^{2}}{2}\right) \left( 1+\frac{r_{3}}{6}s^{3}\frac{1}{%
\sqrt{n}}+o\left( \frac{s^{3}}{\sqrt{n}}\right) \right) 
\label{limitcalculationv2}
\end{equation}%
holds uniformly for $|s| < n^{\varepsilon}$ with $0< \varepsilon < 1/6$.
In order to prove the statement, we use the Fourier transforms and the usual
estimations%
\begin{eqnarray*}
&&\bigg\Vert \sqrt{n}\overset{\pi }{\underset{-\pi }{\int }}\exp \left(
-ixt\right) e_{j}^{T}\alpha ^{n}\left( t\right) dt \\
&&-\mu ^{T}\frac{\sqrt{2\pi }%
}{\sigma }\exp \left( -\frac{x^{2}}{2n\sigma ^{2}}\right) \left[ 1-\frac{%
ir_{3}}{6}x\left( 3\sigma ^{2}n-x^{2}\right) \frac{1}{\sigma ^{6}}\frac{1}{%
n^{2}}\right] \bigg\Vert \\
&\leq &\underset{\left\vert s\right\vert <n^{\varepsilon }}{\int }\left\Vert
e_{j}^{T}p(0)\lambda ^{n}\left( \frac{s}{\sqrt{n}}\right) -\mu ^{T}\exp (-\frac{%
\sigma ^{2}s^{2}}{2})\left( 1+\frac{r_{3}}{6}\frac{s^{3}}{\sqrt{n}}\right)
\right\Vert ds + o\left( \frac{1}{\sqrt n} \right)\\
&&+ c \left\Vert \mu \right\Vert \underset{\left\vert s\right\vert
>n^{\varepsilon }}{\int }(1+s^3) \exp (-\frac{\sigma ^{2}s^{2}}{2})ds+\underset{%
n^{\varepsilon }<\left\vert s\right\vert <\gamma \sqrt{n}}{\int }\left\Vert
e_{j}^{T}\alpha ^{n}\left( \frac{s}{\sqrt{n}}\right) \right\Vert ds \\
&&+\underset{\gamma \sqrt{n}<\left\vert s\right\vert <\pi \sqrt{n}}{\int }%
\left\Vert e_{j}^{T}\alpha ^{n}\left( \frac{s}{\sqrt{n}}\right) \right\Vert
ds \\
&=&I_{1}+o\left( \frac{1}{\sqrt n} \right)+I_{2}+I_{3}+I_{4},
\end{eqnarray*}%
where $0<\varepsilon <\frac{1}{6}$ is arbitrary. 
The term $o\left( \frac{1}{\sqrt n} \right)$ is the contribution of the 
terms $\frac{s}{\sqrt n}p'(0) + O(\frac{s^2}{ n}) $ in (\ref{alfant}), as we can see that
\begin{eqnarray*}
 &&\underset{|s|<n^{\varepsilon}}{\int} \lambda ^n \left( \frac{s}{ \sqrt n}
  \right) \frac{s}{ \sqrt n}p'(0) ds \\
   &=& \underset{|s|<n^{\varepsilon}}{\int} \exp
    \left(-\frac{\sigma ^{2}s^{2}}{2}\right)  \frac{s}{ \sqrt n}p'(0) ds
      + O \left( \underset{|s|<n^{\varepsilon}}{\int} \exp
       \left(-\frac{\sigma ^{2}s^{2}}{2}\right)
        \frac{s^3}{ \sqrt n} \frac{s}{ \sqrt n}p'(0) ds \right),
	\end{eqnarray*}
	which is $0 + o\left( \frac{1}{\sqrt{n}}\right)$, and
	\begin{equation*}
	\underset{|s|<n^{\varepsilon}}{\int} \left| \lambda ^n \left( \frac{s}{ \sqrt n}
	 \right) \right|  \frac{s^2}{n}  ds
	  = o\left( \frac{1}{\sqrt{n}}\right).
	  \end{equation*}
It is clear that proving $%
I_{j}=o\left( \frac{1}{\sqrt{n}}\right) ,\quad j=1,2,3,4$ is enough for our
purposes. 
(\ref{limitcalculationv2}) yields that the
integrand in $I_{1}$ is equal to $\frac{\delta (n)}{n^{1/2}}s^{3}\exp \left(
-\frac{\sigma ^{2}s^{2}}{2}\right) $, where $\delta (n)\rightarrow 0$
uniformly in $s$. Thus we have $I_{1}=o\left( \frac{1}{\sqrt{n}}\right) $.
It is clear that $I_{2}=o\left( \frac{1}{\sqrt{n}}\right) $, and $I_{4}$
converges exponentially fast to zero. Finally, if $\gamma > 0 $ is small
enough, then
\begin{eqnarray*}
I_{3} =\underset{n^{\varepsilon }<\left\vert s\right\vert <\gamma \sqrt{n}}%
{\int }\left\Vert e_{j}^{T}\alpha ^{n}\left( \frac{s}{\sqrt{n}}\right)
\right\Vert ds \leq  \underset{%
n^{\varepsilon }<\left\vert s\right\vert <\gamma \sqrt{n}}{\int }\exp \left(
-\frac{\sigma ^{2}s^{2}}{4}\right) ds.
\end{eqnarray*}%
So we have $I_{3}=o\left( \frac{1}{\sqrt{n}}\right) $, too.
\end{proof}

\begin{remark}
In Theorem \ref{limittheorem} for the expression subtracted from the
appropriate probability we have: 
\begin{eqnarray*}
&&\mu _{k}\frac{1}{\sqrt{2\pi n}\sigma }\exp \left( -\frac{x^{2}}{2n\sigma
^{2}}\right) \left[ 1-\frac{ir_{3}}{6}x\left( 3\sigma ^{2}n-x^{2}\right) 
\frac{1}{\sigma ^{6}}\frac{1}{n^{2}}\right] \\
&=&\mu _{k}\frac{1}{\sqrt{2\pi n}\sigma }\exp \left( -\frac{y^{2}}{2}\right)
+\mu _{k}\frac{1}{\sqrt{n}}\frac{1}{\sigma }\frac{q_{1}\left( y\right) }{%
\sqrt{n}},
\end{eqnarray*}%
where $y=\frac{x}{\sqrt{n}\sigma }$, and the $q_{1}\left( y\right) $ is the
function defined in \cite{Petrov}, Chapter VI. (1.14). In this sense, the
local limit theorem concerning RWwIS is analogous to the one of Simple
Symmetric Random Walk (see \cite{Petrov} Chapter VII. Theorem 13).
\end{remark}

The extension of Theorem \ref{limittheorem} to the multidimensional case is
straightforward. Analogously to (\ref{largesteigenvalue}), we have:%
\begin{equation*}
\lambda (t)=1-\frac{1}{2}t^{T}\sigma t+f\left( t\right) +o(\left\vert
t\right\vert ^{3})\quad \left( \left\vert t\right\vert \rightarrow 0\right) ,
\end{equation*}%
where $f(t)=\underset{i=1}{\overset{d}{\sum }}\underset{j=1}{\overset{d}{%
\sum }}\underset{k=1}{\overset{d}{\sum }}r_{3,i,j,k}t_{i}t_{j}t_{k}$ is the
third term of the Taylor expansion. Denote%
\begin{equation*}
\Omega =n^{d/2}P(\xi _{n}=(x,.)|\xi
_{0}=(0,j))=\frac{n^{d/2}}{\left( 2\pi \right) ^{d}}\overset{\pi }{\underset{%
-\pi }{\int }}...\overset{\pi }{\underset{-\pi }{\int }}\exp \left(
-i\left\langle x,t\right\rangle \right) e_{j}^{T}\alpha ^{n}\left( t\right)
dt.
\end{equation*}%
So the analogue of the expression subtracted from the appropriate
probability in Theorem \ref{limittheorem} (multiplied by $\frac{n^{d/2}}{%
\left( 2\pi \right) ^{d}}$) is%
\begin{equation*}
I^{\left( n\right) }:=\overset{\infty }{\underset{-\infty }{\int }}...%
\overset{\infty }{\underset{-\infty }{\int }}\exp \left( -\frac{s\sigma s}{2}%
-i\left\langle x,\frac{s}{\sqrt{n}}\right\rangle \right) \frac{f\left(
s\right) }{\sqrt{n}}ds.
\end{equation*}%
Using Lebesgue's Theorem, it is easy to see that $I^{\left( n\right)
}=O\left( n^{-1/2}\right) $. One can estimate $I_{1},I_{2},I_{3},I_{4}$ the
same way, as it was done in the proof of Theorem \ref{limittheorem} (see 
\cite{Kramli-Szasz} Section 5. for more details). So we have arrived at

\begin{proposition}
\label{limitpropmoredimension} Supposing that (\ref{harmadikmomentum}) exists, for a $d$ dimensional RWwIS %
\begin{equation*}
P(\xi _{n}=(x,k)|\xi _{0}=(0,j))=\frac{1}{n^{d/2}}\mu _{k}g_{\sigma }\left( 
\frac{x}{\sqrt{n}}\right) +O\left( n^{-\left( d+1\right) /2}\right) 
\end{equation*}%
holds, where $g_{\sigma }\left( x\right) $ denotes the density of a Gaussian
distribution with mean $0$ and covariance matrix $\sigma $ and the great order
is uniform is $x$.
\end{proposition}

A further refinement of the local limit theorem will be useful in the
sequel. Now, we would like to go further in the asymptotic expansion, and
apply our techniques in the two dimensional case. Nevertheless, we are
interested only in an estimation, not in the exact result which will
simplify the calculation. Just like previously, let us begin with the one
dimensional case. Assume the convergence of the series%
\begin{equation}
\Upsilon =\underset{y\in 
\mathbb{Z}
}{\sum }y^{4}A_{y}.  \label{negyedikmomentum}
\end{equation}%
Now, just like previously, we may write%
\begin{equation*}
\alpha (t)=Q+itM-\frac{t^{2}}{2}\Sigma -\frac{it^{3}}{6}\Xi +\frac{t^{4}}{24}%
\Upsilon +o(t^{4})\quad (t\rightarrow 0)
\end{equation*}%
for the Fourier transform, and%
\begin{equation}
\lambda (t)=1+r_{1}t-\frac{\sigma ^{2}}{2}t^{2}+\frac{r_{3}}{6}%
t^{3}+O(t^{4})\quad (t\rightarrow 0)  \label{apr16nulla}
\end{equation}%
for the largest eigenvalue of $\alpha (t)$. As previously, we have%
\begin{eqnarray*}
&&\left( 1-\frac{\sigma ^{2}s^{2}}{2n}+\frac{r_{3}}{6}\frac{s^{3}}{n^{\frac{3}{%
2}}}+O\left( \frac{s^{4}}{n^{2}}\right) \right) ^{n} \\
&=&\exp \left( -\frac{%
\sigma ^{2}s^{2}}{2}\right) \left( 1+\frac{r_{3}}{6}s^{3}\frac{1}{\sqrt{n}}%
+O\left( \frac{s^{4}+s^{6}}{n}\right) \right) 
\end{eqnarray*}%
uniformly for $|s| < n^{\varepsilon}$. A very similar argument to the previous one (with $I_{j}=o\left( \frac{1}{n}%
\right) ,\quad j=1,2,3,4$) leads to%
\begin{eqnarray}
&&P(\xi _{n} =(x,k)|\xi _{0}=(0,j))  \label{apr16egyyy} \\
&=&\mu _{k}\frac{1}{\sqrt{2\pi n}\sigma }\exp \left( -\frac{x^{2}}{2n\sigma
^{2}}\right) \left[ 1-\frac{ir_{3}}{6}x\left( 3\sigma ^{2}n-x^{2}\right) 
\frac{1}{\sigma ^{6}}\frac{1}{n^{2}}\right] +O\left( \frac{1}{n^{3/2}}%
\right) ,  \notag
\end{eqnarray}%
where the great order on the right hand side is uniform in $x$.

Now our aim is to formulate an assertion similar to (\ref{apr16egyyy}) in
two dimensions. Applying the one dimensional proof to the two dimensional
case it is easily seen that%
\begin{equation*}
P(\xi _{n}=(x,k)|\xi _{0}=(0,j))-\mu _{k}\frac{1}{n}g_{\sigma }\left( \frac{x%
}{\sqrt{n}}\right) +F(x,n)=O\left( \frac{1}{n^{2}}\right) ,
\end{equation*}%
where the great order is again uniform in $x$, and $F(x,n)$ is equal to%
\begin{equation*}
\frac{1}{2\pi n^{3/2}}\underset{i_{1}=1}{\overset{2}{\sum }}\underset{%
i_{2}=1}{\overset{2}{\sum }}\underset{i_{3}=1}{\overset{2}{\sum }}\underset{%
-\infty }{\overset{\infty }{\int }}\underset{-\infty }{\overset{\infty }{%
\int }}\exp \left( -\frac{s^{T}\sigma s}{2}-i\left\langle x,\frac{s}{\sqrt{n%
}}\right\rangle \right) r_{3,i_{1},i_{2},i_{3}}s_{i_{1}}s_{i_{2}}s_{i_{3}}ds.
\end{equation*}%
We estimate $F(x,n)$ just like it was done in \cite{Pene b}. Observe that
with the notation%
\begin{equation*}
\Psi (x)=\underset{-\infty }{\overset{\infty }{\int }}\underset{-\infty }{%
\overset{\infty }{\int }}\exp \left( -\frac{s^{T}\sigma s}{2}-i\left\langle
x,s\right\rangle \right) ds=\frac{2\pi }{\sqrt{\left\vert \sigma \right\vert 
}}\exp \left( -\frac{x^{T}\sigma ^{-1}x}{2}\right) ,
\end{equation*}%
we have with an appropriate $C_{1}$\ constant%
\begin{equation*}
\left\vert F(x,n)\right\vert <C_{1}\frac{1}{n^{3/2}}\underset{%
i_{1},i_{2},i_{3}}{\max }\left\vert \frac{\partial ^{3}\Psi }{\partial
x_{i_{1}}\partial x_{i_{2}}\partial x_{i_{3}}}\left( \frac{x}{\sqrt{n}}%
\right) \right\vert .
\end{equation*}%
Further, observe that%
\begin{equation*}
\left\vert \frac{\partial ^{3}\Psi }{\partial x_{i_{1}}\partial
x_{i_{2}}\partial x_{i_{3}}}\left( x\right) \right\vert <C_{2}\left(
\left\Vert x\right\Vert +\left\Vert x\right\Vert ^{3}\right) \exp \left( -%
\frac{x^{T}\sigma ^{-1}x}{2}\right) .
\end{equation*}%
So we have arrived at

\begin{proposition}
\label{Penelemma}Assume that for a two dimensional RWwIS (\ref%
{negyedikmomentum}) exists. Then there is a $C$ constant, such that for
every $x\in 
\mathbb{R}
^{2}$ and for every $1\leq j,k\leq s$\ the following estimation holds%
\begin{eqnarray*}
&&\left\vert P(\xi _{n}=(x,k)|\xi _{0}=(0,j))-\mu _{k}\frac{1}{n}g_{\sigma
}\left( \frac{x}{\sqrt{n}}\right) \right\vert \\
&\leq &C\left( \frac{1}{n^{3/2}}\left( \frac{\left\Vert x\right\Vert }{%
n^{1/2}}+\frac{\left\Vert x\right\Vert ^{3}}{n^{3/2}}\right) \exp \left( -%
\frac{x^{T}\sigma ^{-1}x}{2n}\right) +\frac{1}{n^{2}}\right) .
\end{eqnarray*}
\end{proposition}

By an elementary argument (see, for instance in \cite{Imbragimov} Theorem 4.2.2), 
using Proposition \ref{Penelemma} one can easily deduce

\begin{corollary}
\label{cor2010}Under the conditions of Proposition \ref{Penelemma}
\[ \underset{x \in \mathbb{Z} ^2}{\sum }
\left\vert P(\eta _{n}=x|\eta _{0}=0)-\frac{1}{n}g_{\sigma
}\left( \frac{x}{\sqrt{n}}\right) \right\vert = O\left( n^{-1/4} \right).
\]
\end{corollary}

\subsection{Reversed walks}
\label{subrev}

The so-called reversed walk will be important in the sequel. If a RWwIS is
given with the appropriate $\left( p_{y,i,j}\right) $ probabilities, then we
define the $\left( q_{y,i,j}\right) $ reversed random walk for which%
\begin{equation}
q_{y,i,j}=\frac{\mu _{j}p_{-y,j,i}}{\mu _{i}}.  \label{reversedwalk}
\end{equation}%
Obviously, the stationary measure of the reversed walk is also $\mu $. As we
would like to apply the local limit theorem for the reversed walk, we need

\begin{proposition}
\label{Proprevw}If the primary RWwIS fulfills our basic assumptions, then
the reversed walk fulfills them as well. Furthermore, the so-called
covariance matrix of the reversed walk is the same as the one of the primary
walk.
\end{proposition}

\begin{proof}
Basic assumptions (i)-(iii) are fulfilled obviously. So it suffices to prove
the second statement. Let us introduce some notations%
\begin{eqnarray*}
\widetilde{A}_{y} &=&\left( q_{y,j,k}\right) _{j,k=1,...,s}, \\
\widetilde{Q} &=&\underset{y\in 
\mathbb{Z}      
^{d}}{\sum }\widetilde{A}_{y}, \\
\widetilde{M}_{l} &=&\underset{y\in 
\mathbb{Z}
^{d}}{\sum }y_{l}\widetilde{A}_{y}, \\
\widetilde{\Sigma }_{l,m} &=&\underset{y\in 
\mathbb{Z}
^{d}}{\sum }y_{l}y_{m}\widetilde{A}_{y},
\end{eqnarray*}%
and a new inner product%
\begin{eqnarray*}
\left( ,\right) &:&%
\mathbb{R}
^{s}\times 
\mathbb{R}
^{s}\rightarrow 
\mathbb{R}
, \\
\left( u,v\right) &=&\overset{s}{\underset{i=1}{\sum }}\mu _{i}u_{i}v_{i}.
\end{eqnarray*}%
Let us denote by $A^{\ast }$\ the adjoint of the linear operator $A$, i.e. $%
\left( u,Av\right) =\left( A^{\ast }u,v\right) $ for all $u,v\in 
\mathbb{R}
^{s}$. Elementary calculations show that $\widetilde{Q}=Q^{\ast }$, $%
\widetilde{A}_{y}=\left( A_{-y}\right) ^{\ast }$, $\widetilde{M}_{l}=-\left(
M_{l}\right) ^{\ast }$, $\widetilde{\Sigma }_{l,m}=\left( \Sigma
_{l,m}\right) ^{\ast }$ for all $y\in 
\mathbb{Z}
^{d},1\leq l,m\leq s$. Now, for an arbitrary element $\widetilde{\sigma }%
_{l,m}$ of the "covariance matrix" defined for the reversed walk%
\begin{eqnarray*}
\widetilde{\sigma }_{l,m} &=&\left( 1,\widetilde{\Sigma }_{l,m}1\right)
-\left( 1,\widetilde{M}_{l}\left( \widetilde{Q}-1\right) ^{-1}\widetilde{M}%
_{m}1\right) -\left( 1,\widetilde{M}_{m}\left( \widetilde{Q}-1\right) ^{-1}%
\widetilde{M}_{l}1\right) \\
&=&\left( \Sigma _{l,m}1,1\right) -\left( M_{m}\left( Q-1\right)
^{-1}M_{l}1,1\right) -\left( M_{l}\left( Q-1\right) ^{-1}M_{m}1,1\right) \\
&=&\sigma _{l,m}.
\end{eqnarray*}%
Hence the statement.
\end{proof}

\section{Visited points in high dimensions}
\label{Secthighdim}

In the high dimensional case, we find that $E_{d}\left( n\right) $ grows
fast, i.e. linearly in $n$, as we could have conjectured it from the
transiency of the RWwIS. In Theorem \ref{highdimEdn} we prove this fact and
compute remainder terms, too. Our approach is based on the one of \cite%
{Dv-Er}, but there are some main differences. First, we have to consider the
reversed random walk which is trivial in the case of \cite{Dv-Er}. After it,
the renewal equation is written with matrices and vectors, which is more
technical than in the case of \cite{Dv-Er}. Moreover, there will be a
technical difficulty, namely we will have to consider the case, when the
distribution of $\varepsilon _{0}$ is arbitrary. This will be treated
separately in Proposition \ref{highdimEdnprop}. After it, we will be able to
estimate $V_{d}\left( n\right) $. In fact, $o\left( n^{2}\right) $ is enough
for proving weak law of large numbers, and $O\left( n^{2-\delta }\right) $
for strong law of large numbers, but our estimations will be sharper.
Nevertheless, these estimations are weaker than the ones of \cite{Dv-Er}
because a symmetry argument, used in \cite{Dv-Er}, fails here. That is why
the computation is longer and it uses Proposition \ref{highdimEdnprop}, too.
Let us see the details.

\begin{theorem}
\label{highdimEdn}Let $d\geq 3$. Assuming that $\varepsilon _{0}$ is
distributed according to its unique stationary measure, we have%
\begin{eqnarray*}
E_{3}\left( n\right) &=&n\gamma _{3}+O(\sqrt{n}) \\
E_{4}\left( n\right) &=&n\gamma _{4}+O(\log n) \\
E_{d}\left( n\right) &=&n\gamma _{d}+\beta _{d}+O(n^{2-d/2})\quad \text{\
for }d\geq 5\text{ }
\end{eqnarray*}
with some constants $\gamma _{d},\beta _{d}$, depending on the RWwIS.
\end{theorem}

\begin{proof}
Fix some dimension $d\geq 3$. For the sake of simplicity, we skip the
index $d$ and denote $E_d (n)=\underset{k=1}{\overset{n}{\sum }}\gamma (k)$. Consider an $\left\{ \xi _{k}=\left( \eta
_{k},\varepsilon _{k}\right) ,0\leq k\right\} $ RWwIS fulfilling our
assumptions. Let $\left\{ \widetilde{\xi }_{k}=\left( \widetilde{\eta }_{k},%
\widetilde{\varepsilon }_{k}\right) ,0\leq k\right\} $ be the reversed walk,
i.e. for which the transition probabilities are defined by (\ref%
{reversedwalk}). Put $\eta _{0}=0$, $\gamma (0)=1$ and define%
\begin{equation*}
\gamma (n)=P\left( \eta _{n}\notin \left\{ \eta _{0},...,\eta _{n-1}\right\}
\right)
\end{equation*}%
which is just the probability that the walk visits a new point at step $n$.
Obviously%
\begin{eqnarray*}
\gamma (n) &=&P(\eta _{i}\neq \eta _{n}\quad i=0,..n-1) \\
&=&P(\eta _{n}-\eta _{i}\neq 0\quad i=0,..n-1) \\
&=&P(\widetilde{\eta }_{n-i}\neq 0\quad i=0,..n-1) \\
&=&P(\widetilde{\eta }_{j}\neq 0\quad j=1,..n).
\end{eqnarray*}

It is clear that we have to examine\ the reversed walk.

Define $U_{k}\in 
\mathbb{R}
^{s\times s}$ with 
\begin{equation*}
\left( U_{k}\right) _{i,j}=P\left( \widetilde{\xi }_{k}=(0,j)|\widetilde{\xi 
}_{0}=(0,i)\right)
\end{equation*}%
and $R_{k}\in 
\mathbb{R}
^{s}$ with 
\begin{equation*}
\left( R_{k}\right) _{j}=P(0\notin \left\{ \widetilde{\eta }_{1},...,%
\widetilde{\eta }_{k}\right\} |\widetilde{\xi }_{0}=(0,j)).
\end{equation*}
Obviously, we have:%
\begin{equation*}
\overset{n}{\underset{k=0}{\sum }}U_{k}\cdot R_{n-k}=\underline{1}.
\end{equation*}%
We are interested in $\left\langle R_{n},\mu \right\rangle =\gamma (n)$.
From the definition of $R_{k}$, for $n_{1}>n_{2}$ we have $%
R_{n_{2}}-R_{n_{1}}\geq \underline{0}$, which means that all the components
of the vector are non-negative.

We know from Proposition \ref{Proprevw} and \cite{Kramli-Szasz} Theorem 5.2.
that $\left( U_{k}\right) _{i,j}=c_{j}k^{-\frac{d}{2}}+o_{i,j}(k^{-\frac{d}{2%
}})$. Here we have $c_{j}=c\mu _{j}$, but this fact will not be used. So we
have%
\begin{equation*}
\left( \overset{n}{\underset{k=0}{\sum }}U_{k}\right) _{i,j}=\widetilde{c}%
_{i,j}+O\left( n^{1-\frac{d}{2}}\right) .
\end{equation*}%
Using the monotonity of $R_{k}$ we infer 
\begin{equation*}
\underline{1}\geq \left( \overset{n}{\underset{k=0}{\sum }}U_{k}\right)
\cdot R_{n}.
\end{equation*}%
Defining $\widehat{c}_{j}$ the following way%
\begin{equation*}
\left( \left( \frac{1}{s}\underline{1}\right) ^{T}\cdot \left( \overset{n}{%
\underset{k=0}{\sum }}U_{k}\right) \right) _{j}=\frac{1}{s}\underset{i=1}{%
\overset{s}{\sum }}\left( \widetilde{c}_{i,j}+O\left( n^{1-\frac{d}{2}%
}\right) \right) =\widehat{c}_{j}+O\left( n^{1-\frac{d}{2}}\right) ,
\end{equation*}%
we have%
\begin{equation}
1\geq \left\langle \left( \widehat{c}_{1}+O\left( n^{1-\frac{d}{2}}\right)
,...,\widehat{c}_{s}+O\left( n^{1-\frac{d}{2}}\right) \right)
,R_{n}\right\rangle .  \label{highdimupperbound}
\end{equation}%
For all $j$, $\left( R_{n}\right) _{j}$ has a limit in $n$, being a
decreasing non-negative sequence. So write $\left( R_{n}\right)
_{j}=R^{j}+a_{n}^{j}$, where $a_{n}^{j}\searrow 0$. It will be enough to
estimate the order of $a_{n}^{j}$, because $\gamma (n)=\underset{j=1}{%
\overset{s}{\sum }}\mu _{j}\left( R^{j}+a_{n}^{j}\right) $.

For the estimation of the other direction let $k<n$. We have:%
\begin{equation*}
\left( \frac{1}{s}\underline{1}\right) ^{T}\cdot \left( \overset{k}{\underset%
{i=0}{\sum }}U_{i}\right) \cdot R_{n-k}+\left( \frac{1}{s}\underline{1}%
\right) ^{T}\cdot \left( \overset{n}{\underset{i=k+1}{\sum }}U_{i}\right)
\cdot \underline{1}\geq 1.
\end{equation*}%
Since $\left( U_{k}\right) _{i,j}\geq 0$ for all $k,i,j$, we have $\left( 
\frac{1}{s}\underline{1}\right) ^{T}\cdot \left( \overset{k}{\underset{i=0}{%
\sum }}U_{i}\right) \leq \left( \widehat{c}_{1},...,\widehat{c}_{s}\right) $%
. On the other hand, $\left( \frac{1}{s}\underline{1}\right) ^{T}\cdot
\left( \overset{n}{\underset{i=k+1}{\sum }}U_{i}\right) \cdot \underline{1}%
=o(1)$, as $k\rightarrow \infty $, thus%
\begin{equation}
\left\langle \left( \widehat{c}_{1},...,\widehat{c}_{s}\right)
,R_{n-k}\right\rangle \geq 1+o(1).  \label{hoghdimlowerbound}
\end{equation}%
So if we let $n\rightarrow \infty $, $k\rightarrow \infty $, $n-k\rightarrow
\infty $, (\ref{hoghdimlowerbound}) together with (\ref{highdimupperbound})
yields%
\begin{equation*}
\widehat{c}_{1}R^{1}+...+\widehat{c}_{s}R^{s}=1.
\end{equation*}%
Substituting to (\ref{highdimupperbound}) we have:%
\begin{equation*}
\overset{s}{\underset{j=1}{\sum }}\left[ \widehat{c}_{j}a_{n}^{j}+O\left(
n^{1-\frac{d}{2}}\right) R^{j}+O\left( n^{1-\frac{d}{2}}\right) a_{n}^{j}%
\right] \leq 0,
\end{equation*}%
whence%
\begin{equation*}
\overset{s}{\underset{j=1}{\sum }}\widehat{c}_{j}a_{n}^{j}\leq O\left( n^{1-%
\frac{d}{2}}\right) .
\end{equation*}%
Since $\widehat{c}_{j}>0$ and $a_{n}^{j}\geq 0$, we conclude that $%
a_{n}^{j}=O\left( n^{1-\frac{d}{2}}\right) $ for $1\leq j\leq s$. This
yields $\gamma (n)=\underset{j=1}{\overset{s}{\sum }}\mu _{j}\left(
R^{j}+a_{n}^{j}\right) =\gamma +O\left( n^{1-\frac{d}{2}}\right) $. Hence
the statement (just like in \cite{Dv-Er}).
\end{proof}

\begin{proposition}
\label{highdimEdnprop}The assertion of Theorem \ref{highdimEdn} remains true
when the distribution of $\varepsilon _{0}$ is arbitrary.
\end{proposition}

\begin{proof}
With the notation $\gamma (n)=\gamma +h(n)$ we already know that $%
h(n)=O\left( n^{1-\frac{d}{2}}\right) $. Define $\gamma ^{e_{j}}(n)=P\left(
\eta _{n}\notin \left\{ \eta _{0},...,\eta _{n-1}\right\} |\varepsilon
_{0}=j\right) $ and $\gamma ^{e_{j}}(n)=\gamma +h^{j}(n)$ for $j=1,...,s$.
As in the previous proof, it would be sufficient to prove $h^{j}(n)=O\left(
n^{1-\frac{d}{2}}\right) $ for all $j$.

For the present, let $K$ be a fixed, great natural number, and%
\begin{equation*}
\mu _{k}+b_{k}^{j}(K)=P\left( \varepsilon _{K}=k|\varepsilon _{0}=j\right)
\quad j,k=1,...s.
\end{equation*}

We know from the ergodic theorem of Markov chains that $b_{k}^{j}(K)$ tends
to zero exponentially fast in $K$.

Denote by $p(K,n)$ the probability of visiting such a site at time $n$ that
was visited during the first $K$ steps, but was not visited in the following 
$(n-K-1)$ steps, provided that $\varepsilon _{0}=j$. We know from \cite%
{Kramli-Szasz} Theorem 5.2. that $p(K,n)=O\left( K\cdot \left( n-K\right) ^{-%
\frac{d}{2}}\right) $, whence%
\begin{equation}
\gamma ^{e_{j}}(n)=\overset{s}{\underset{k=1}{\sum }}\left[ \left( \mu
_{k}+b_{k}^{j}(K)\right) \gamma ^{e_{k}}(n-K)\right] +O\left( K\cdot \left(
n-K\right) ^{-\frac{d}{2}}\right) .  \label{highprop1}
\end{equation}

Recall $\gamma ^{e_{j}}(n)=\gamma +h^{j}(n)$ to infer that $h^{j}(n)$ is equal to%
\begin{eqnarray}
\overset{s}{\underset{k=1}{\sum }}\mu _{k}h^{k}(n-K)+\overset{s}{%
\underset{k=1}{\sum }}b_{k}^{j}(K)h^{k}(n-K)+O\left( K\cdot \left(
n-K\right) ^{-\frac{d}{2}}\right) \notag \\ 
=:I+II+III.  \label{highprop2}
\end{eqnarray}

Now, put $K=K(n)=\left\lfloor n^{\alpha }\right\rfloor $ with arbitrary $%
0<\alpha <1$. It is clear that $I$ is equal to $h(n-K)$, so the proof of
Theorem \ref{highdimEdn} yields $I=O\left( \left( n-n^{\alpha }\right) ^{1-%
\frac{d}{2}}\right) \leq O\left( n^{1-\frac{d}{2}}\right) $. Since $%
b_{k}^{j}(K)$\ tends to zero exponentially fast in $K$ we have $II\leq
O\left( n^{1-\frac{d}{2}}\right) $. Finally, $III=O\left( n^{\alpha }\left(
n-n^{\alpha }\right) ^{-\frac{d}{2}}\right) \leq O\left( n^{1-\frac{d}{2}%
}\right) $. Hence the statement.
\end{proof}

Now, let us see the estimation of $V_{d}\left( n\right) $.

\begin{theorem}
\label{highdimvdn}For $d\geq 3$ assuming that $\varepsilon _{0}\sim \mu $ we
have%
\begin{equation*}
V_{d}(n)=O\left( n^{1+\frac{2}{d}}\right) .
\end{equation*}
\end{theorem}

\begin{proof}
Let $\gamma \left( n,m\right) $ denote the probability that the RWwIS visits
new points in both the $n^{th}$ and the $m^{th}$ step under the condition
that $\varepsilon _{0}\sim \mu $, and let $A=\{\eta _{i}\neq \eta _{m},\quad
i=0,...,m-1\}$. Obviously, $\gamma _{d}\left( n,m\right) =\gamma _{d}\left(
m,n\right) $, so, when estimating $\gamma \left( n,m\right) $ one can assume 
$n>m$.%
\begin{eqnarray*}
\gamma \left( m,n\right) &=&P\left( A\text{ \& }\eta _{j}\neq \eta
_{n},\quad j=0,...,n-1\right) \\
&\leq &P\left( A\text{ \& }\eta _{j}\neq \eta _{n},\quad j=m,...,n-1\right)
\\
&=&\gamma (n)P\left( \eta _{j}\neq \eta _{n},\quad j=m,...,n-1\quad |\quad
A\right) .
\end{eqnarray*}

Here, $P\left( \eta _{j}\neq \eta _{n},\quad i=m,...,n-1\quad |\quad
A\right) $ is the probability that the RWwIS visits a new point in the $%
(n-m)^{th}$ step, assuming that the distribution of $\varepsilon _{0}$ is
some $\mu \left( n\right) $. So the condition $A$ is involved in $\mu \left(
n\right) $, and because of the Markov property, it has no other
contribution. The probability of this event is denoted by $\gamma _{d}^{\mu
(n)}(n-m)$. Because of Proposition \ref{highdimEdnprop} we know that $\gamma
_{d}^{\mu (n)}(n-m)\rightarrow \gamma _{d}$, as $(n-m)\rightarrow \infty $,
and it is easy to see that this convergence is uniform in $\mu \left(
n\right) $. So we know that for $\forall \delta >0$ $\exists $ $N=N\left(
\delta \right) $, such that for $\forall n-m>N$ the following estimation
holds. 
\begin{equation*}
\gamma _{d}^{\mu (n)}(n-m)=\underset{j=1}{\overset{s}{\sum }}\mu
(n)_{j}\gamma _{d}^{e_{j}}(n-m)<(1+\delta )\gamma _{d}(n-m).
\end{equation*}

In addition, using Proposition \ref{highdimEdnprop}, one can estimate $%
N\left( \delta \right) $, which will be done a little bit later. Now, let us
see the estimation of $V_{d}\left( n\right) $%
\begin{eqnarray*}
V_{d}\left( n\right) &=&\underset{i,j=0}{\overset{n}{\sum }}\gamma _d \left(
i,j\right) -\underset{i=0}{\overset{n}{\sum }}\gamma _d \left( i\right) 
\underset{j=0}{\overset{n}{\sum }}\gamma _d \left( j\right) \\
&\leq &2\underset{0\leq i\leq j\leq n}{\sum }\left( \gamma _d \left( i,j\right)
-\gamma _d \left( i\right) \gamma _d \left( j\right) \right) \\
&\leq &2\underset{0\leq i<i+K\leq j\leq n}{\sum }\left( \gamma _d \left(
i,j\right) -\gamma _d \left( i\right) \gamma _d \left( j\right) \right) +2\underset%
{\underset{i\leq j<i+K}{0\leq i\leq n}}{\sum }\gamma _d \left( i,j\right) \\
&=&:S_{1}+S_{2}.
\end{eqnarray*}

Let $K$ be big enough, such that for $n-m>K$ one would have $\gamma _d ^{\nu
}(n-m)<(1+\delta )\gamma _d (n-m)$ for arbitrary $\nu $. Estimating $S_{1}$ and 
$S_{2}$ separately, we get%
\begin{eqnarray*}
\frac{S_{1}}{2} &=&\underset{i=0}{\overset{n-K}{\sum }}\underset{j=i+K}{%
\overset{n}{\sum }}\gamma _{d}\left( i,j\right) -\underset{i=0}{\overset{n-K}%
{\sum }}\underset{j=i}{\overset{n}{\sum }}\gamma _{d}\left( i\right) \gamma
_{d}\left( j\right) +\underset{i=0}{\overset{n-K}{\sum }}\underset{j=i}{%
\overset{i+K}{\sum }}\gamma _{d}\left( i\right) \gamma _{d}\left( j\right) \\
&\leq &\underset{i=0}{\overset{n-K}{\sum }}\gamma _{d}\left( i\right) 
\underset{0\leq i\leq n-K}{\max }\left( \underset{j=i}{\overset{n}{\sum }}%
\left( 1+\delta \right) \gamma _{d}\left( j-i\right) -\underset{j=i}{\overset%
{n}{\sum }}\gamma _{d}\left( j\right) \right) \\
&&+\underset{i=0}{\overset{n-K}{\sum }}\gamma _{d}\left( i\right) \underset{%
j=i}{\overset{i+K}{\sum }}\gamma _{d}\left( j\right),
\end{eqnarray*}
which can be bounded by
\begin{eqnarray*}
&\leq &\underset{i=0}{\overset{n-K}{\sum }}\gamma _{d}\left( i\right) \left[
\delta E_{d}(n)+E_{d}(n-\left\lfloor \frac{n}{2}\right\rfloor
)-E_{d}(n)+E_{d}(\left\lfloor \frac{n}{2}\right\rfloor )\right] \\
&&+\underset{i=0}{\overset{n-K}{\sum }}\gamma _{d}\left( i\right) K.
\end{eqnarray*}%
On the other hand,%
\begin{equation*}
S_{2}\leq 2\underset{\underset{i\leq j<i+K}{0\leq i\leq n}}{\sum }\gamma
\left( i\right) \leq 2KE_{d}(n).
\end{equation*}%
From the proof of Proposition \ref{highdimEdnprop}, one can easily deduce
that for $k$ large enough
\begin{equation*}
\gamma _{d}^{\nu }(k)<\left( 1+O(k^{1-\frac{d}{2}})\right) \gamma
_{d}(k), 
\end{equation*}
uniformly in $\nu $. So replacing $K$ to $K\left( n\right) $ in
the above argument, one can change $\delta $ to $O\left( K(n)^{1-\frac{d}{2}%
}\right) $, thus%
\begin{eqnarray*}
V_{3}\left( n\right) &\leq &O(n)\left[ O\left( K(n)^{1-\frac{d}{2}}\right)
O\left( n\right) +O\left( \sqrt{n}\right) \right] +K\left( n\right) O\left(
n\right) \\
V_{4}\left( n\right) &\leq &O(n)\left[ O\left( K(n)^{1-\frac{d}{2}}\right)
O\left( n\right) +O\left( \log n\right) \right] +K\left( n\right) O\left(
n\right) \\
V_{d}\left( n\right) &\leq &O(n)\left[ O\left( K(n)^{1-\frac{d}{2}}\right)
O\left( n\right) +O\left( 1\right) \right] +K\left( n\right) O\left(
n\right) \quad d\geq 5.
\end{eqnarray*}%
Now, the choice $K(n)=\left\lfloor n^{\frac{2}{d}}\right\rfloor $ completes
the proof.
\end{proof}

\begin{proposition}
\label{highdimvdnprop}The assertion of Theorem \ref{highdimvdn} remains true
when the distribution of $\varepsilon _{0}$ is some arbitrary $\nu $.
Moreover, the great order is uniform in $\nu $.
\end{proposition}

\begin{proof}
Let us introduce the notation $E ^{\nu
}\left[ . \right] $ for the expectation when $%
\varepsilon _{0}\sim \nu $. For convenience, we also write 
$E_d ^{\nu }\left( n \right) $ and $V_d ^{\nu }\left( n \right) $
for the expectation and variance of $L_d \left( n \right) $
when $\varepsilon _{0}\sim \nu $. Obviously,%
\begin{equation}
V_{d}^{\nu }\left( n\right) =E ^{\nu }\left[ \left( L_{d}\left( n\right)
\right) ^{2}\right] -\left( E_{d}^{\nu }\left( n\right) \right) ^{2}.
\label{maj181}
\end{equation}%
On the other hand,%
\begin{equation}
\underset{j=1}{\overset{s}{\sum }}\nu _{j}V_{d}^{e_{j}}\left( n\right) =%
\underset{j=1}{\overset{s}{\sum }}\nu _{j}E ^{e_j } \left[ \left( L_{d}\left(
n\right) \right) ^{2}\right] -\underset{j=1}{\overset{s}{\sum }}\nu
_{j}\left( E_{d}^{e_{j}}\left( n\right) \right) ^{2}.  \label{maj182}
\end{equation}%
Since $E ^{\nu }\left[ \left( L_{d}\left( n\right) \right) ^{2}\right] =%
\underset{j=1}{\overset{s}{\sum }}\nu _{j}E ^{e_j} \left[ \left( L_{d}\left(
n\right) \right) ^{2}\right] $, subtracting (\ref{maj182}) from (\ref{maj181}%
), we conclude%
\begin{eqnarray}
V_{3}^{\nu }\left( n\right) -\underset{j=1}{\overset{s}{\sum }}\nu
_{j}V_{3}^{e_{j}}\left( n\right) &=&O\left( n^{3/2}\right) ,  \label{maj183}
\\
V_{d}^{\nu }\left( n\right) -\underset{j=1}{\overset{s}{\sum }}\nu
_{j}V_{d}^{e_{j}}\left( n\right) &=&O\left( n\log n\right) \quad d\geq 4.
\label{maj184}
\end{eqnarray}%
It is clear that the great order on the right hand side is uniform in $\nu $%
. In the sense of (\ref{maj183}) and (\ref{maj184}) it is enough to prove
the statement for $\nu =e_{j},\left( j=1,...,s\right) $. To do so,
substitute $\mu =\nu $ to (\ref{maj183}) and (\ref{maj184}) and use Theorem %
\ref{highdimvdn}\ to infer%
\begin{equation*}
\underset{j=1}{\overset{s}{\sum }}\mu _{j}V_{3}^{e_{j}}\left( n\right)
=O\left( n^{1+\frac{d}{2}}\right) ,\quad d\geq 3.
\end{equation*}%
Since for all $d,j$ and $n$ $\mu _{j}$ and $V_{d}^{e_{j}}\left( n\right) $
are non negative, we have proved the statement for all $e_{j}$.
\end{proof}

\begin{corollary}
\label{highdimweaklaw}For RWwIS in $d\geq 3$ the weak law of large numbers
holds, namely%
\begin{equation*}
P\left( \left\vert L_{d}\left( n\right) -E_{d}\left( n\right) \right\vert
>\varepsilon E_{d}\left( n\right) \right) \rightarrow 0
\end{equation*}%
for $\forall \varepsilon >0$.
\end{corollary}

\begin{proof}
Since $V_{d}(n)=o\left( n^{2}\right) $ Chebyshev's inequality applies (just
like in \cite{Dv-Er}).
\end{proof}

From Theorem \ref{highdimvdn} one can deduce even strong law of large
numbers:

\begin{theorem}
\label{highdimstronglaw}For RWwIS in $d\geq 3$ strong law of large numbers
holds, namely%
\begin{equation*}
P\left( \underset{n\rightarrow \infty }{\lim }\frac{L_{d}(n)}{E_{d}(n)}%
=1\right) =1.
\end{equation*}
\end{theorem}

Theorem \ref{highdimstronglaw} can be proved almost the same way as it was
done in \cite{Dv-Er}. The difference is that if we have $V_{d}\left(
n\right) =O\left( n^{\tau }\right) $ with some $\tau <2$, then we have to
choose parameters $\alpha $ and $\beta $ to fulfill

\begin{eqnarray*}
\frac{1+\tau }{3} &<&\alpha <1 \\
\frac{1}{2\alpha -\tau } &<&\beta <\frac{1}{1-\alpha }.
\end{eqnarray*}%
After it, the argument of \cite{Dv-Er} works. So the main point is that we
should have some $\tau <2$ such that $V_{d}\left( n\right) =O\left( n^{\tau
}\right) $ as it was mentioned at the beginning of the Section.

Identifying the constant $\gamma $ is an interesting question, though we
cannot give a closed formula in the general case.

We only know that for the constant $\gamma $ we have 
\begin{equation}
\gamma =P(\eta _{k}\neq 0:k\geq 1|\varepsilon _{0}\sim \mu ) .
\label{gammaderteke}
\end{equation}

To see this, first, observe that the constant $\gamma $ is the same for the
primary and the reversed walk. We have seen that%
\begin{equation*}
\gamma (n)=P(\widetilde{\eta }_{j}\neq 0\quad j=1,..n).
\end{equation*}%
Taking $n\rightarrow \infty $, (\ref{gammaderteke}) follows.

\section{Visited points in two dimensions}
\label{Secttwodim}

In this section we calculate $E_{2}\left( n\right) $ and estimate $%
V_{2}\left( n\right) $. The arguments (assuming that $\varepsilon _{0}\sim
\mu $) are similar to the ones of Theorem \ref{highdimEdn} and \ref%
{highdimvdn}, or \cite{Dv-Er} Theorem 1 and Theorem 2. The computations are
longer than in \cite{Dv-Er}. We have to write the renewal equation in terms
of vectors and matrices, which is a new idea, and we use the above proved
Proposition \ref{limitpropmoredimension} because\ it is essential that the
remainder term of the probability of returning to the origin should be
summable, which was trivial in the case of \cite{Dv-Er}. We have to consider
the case of arbitrary initial distribution, separately, just like in Section %
\ref{Secthighdim}. In this case, we formulate the fact that after some steps
the distribution of $\varepsilon $ will be very close to $\mu $.

\begin{theorem}
\label{twodimEdn}Let $d=2$. Assuming that $\varepsilon _{0}\sim \mu $ and
that (\ref{harmadikmomentum}) exists, we have%
\begin{equation*}
E_{2}\left( n\right) =\frac{2\pi \sqrt{\left\vert \sigma \right\vert }n}{%
\log n}+O\left( \frac{n\log \log n}{\log ^{2}n}\right) .
\end{equation*}
\end{theorem}

\begin{proof}
As in the proof of Theorem \ref{highdimEdn}, we examine the reversed RWwIS
and write the renewal equation%
\begin{equation}
\overset{n}{\underset{k=0}{\sum }}U_{k}\cdot R_{n-k}=\underline{1}.
\label{renewelequtwodim}
\end{equation}%
Proposition \ref{limitpropmoredimension} yields%
\begin{equation*}
\left( U_{k}\right) _{i,j}=\frac{1}{2\pi \sqrt{\left\vert \sigma \right\vert 
}}\mu _{j}\frac{1}{k}+O\left( k^{-3/2}\right) ,
\end{equation*}%
thus%
\begin{equation}
\left( \underset{k=0}{\overset{n}{\sum }}U_{k}\right) _{i,j}=\frac{1}{2\pi 
\sqrt{\left\vert \sigma \right\vert }}\mu _{j}\log \left( c_{i,j}n\right)
+O\left( n^{-1/2}\right) .  \label{sumofun}
\end{equation}

Our purpose is to estimate $\left\langle R_{n},\mu \right\rangle =\gamma (n)$%
. Exactly as in the high dimensional case, $R_{n}$ is decreasing, so (\ref%
{renewelequtwodim}) yields%
\begin{equation}
\left( \frac{1}{s}\underline{1}\right) ^{T}\cdot \left( \overset{k}{\underset%
{l=0}{\sum }}U_{l}\right) \cdot R_{n-k}+\left( \frac{1}{s}\underline{1}%
\right) ^{T}\cdot \left( \overset{n}{\underset{l=k+1}{\sum }}U_{l}\right)
\cdot \underline{1}\geq 1.  \label{twodimlowebound1}
\end{equation}%
Let $k\rightarrow \infty $, $n\rightarrow \infty $. The relation between $k$
and $n$ will be fixed later. From (\ref{sumofun}) it follows that%
\begin{equation}
\left[ \left( \frac{1}{s}\underline{1}\right) ^{T}\cdot \left( \overset{k}{%
\underset{l=0}{\sum }}U_{l}\right) \right] _{j}=\frac{1}{2\pi \sqrt{%
\left\vert \sigma \right\vert }}\mu _{j}\log \left( \widehat{c}_{j}k\right)
+O\left( k^{-1/2}\right)  \label{twodimlowebound2}
\end{equation}%
for some $\widehat{c}_{j}$. So we have for $k<n$%
\begin{equation}
\left[ \left( \frac{1}{s}\underline{1}\right) ^{T}\cdot \left( \overset{n}{%
\underset{l=k+1}{\sum }}U_{l}\right) \right] _{j}=\frac{1}{2\pi \sqrt{%
\left\vert \sigma \right\vert }}\mu _{j}\log \frac{n}{k}+O\left(
k^{-1/2}\right) .  \label{twodimloebound3}
\end{equation}%
Substituting (\ref{twodimlowebound2}) and (\ref{twodimloebound3}) to the
left hand side of (\ref{twodimlowebound1}) we get%
\begin{eqnarray}
&&\overset{s}{\underset{j=1}{\sum }}\left[ \frac{1}{2\pi \sqrt{\left\vert
\sigma \right\vert }}\mu _{j}\log \left( \widehat{c}_{j}k\right) +O\left(
k^{-1/2}\right) \right] \left( R_{n-k}\right) _{j} \label{twodimlowebound4} \\
&+&\overset{s}{\underset{j=1}%
{\sum }}\frac{1}{2\pi \sqrt{\left\vert \sigma \right\vert }}\mu _{j}\log 
\frac{n}{k}+O\left( k^{-1/2}\right) .  \notag
\end{eqnarray}%
Put $k=\left\lfloor n-\frac{n}{\log n}\right\rfloor $. This yields $\log
k\sim \log \left( n-k\right) $. Using the fact $\gamma (n-k)=\underset{j=1}{\overset{s%
}{\sum }}\mu _{j}\left( R_{n-k}\right) _{j}$, (\ref{twodimlowebound4}) can
be written as%
\begin{eqnarray}
&&\gamma \left( n-k\right) \left[ \frac{1}{2\pi \sqrt{\left\vert \sigma
\right\vert }}\log k\right] + \label{twodimlowebound45} \\
&&\overset{s}{\underset{j=1}{\sum }}\left[ \frac{%
1}{2\pi \sqrt{\left\vert \sigma \right\vert }}\mu _{j}\log \widehat{c}%
_{j}+O\left( k^{-1/2}\right) \right] \left( R_{n-k}\right) _{j}
+C\log \frac{n}{k}+O\left( k^{-1/2}\right) .  \notag
\end{eqnarray}%
Since $\log \frac{n}{k}\rightarrow 0$, and $\left( R_{n-k}\right)
_{j}\rightarrow 0$, as $n-k\rightarrow \infty $ (the latter is the
recurrence property of the two dimensional RWwIS, which is proved in \cite%
{Telcs}), it follows that%
\begin{equation}
\gamma \left( n-k\right) \geq \frac{2\pi \sqrt{\left\vert \sigma \right\vert 
}}{\log k}+o\left( \frac{1}{\log k}\right) .  \label{twodimlowebound5}
\end{equation}%
Hence, by the choice of $k$,%
\begin{equation}
\gamma \left( n-k\right) \geq \frac{2\pi \sqrt{\left\vert \sigma \right\vert 
}}{\log \left( n-k\right) }+o\left( \frac{1}{\log \left( n-k\right) }\right)
.  \label{twodimlowebound6}
\end{equation}%
Now let us give an upper estimation to $\gamma (n)$. From (\ref%
{renewelequtwodim}) it follows that%
\begin{equation*}
\left( \overset{n}{\underset{k=0}{\sum }}U_{k}\right) \cdot R_{n}\leq 
\underline{1}.
\end{equation*}%
Multiplying by the vector $\frac{1}{s}\underline{1}$, we get%
\begin{equation*}
\overset{s}{\underset{j=1}{\sum }}\left[ \frac{1}{2\pi \sqrt{\left\vert
\sigma \right\vert }}\mu _{j}\log \left( \widehat{c}_{j}n\right) +O\left(
n^{-1/2}\right) \right] \left( R_{n}\right) _{j}\leq 1,
\end{equation*}%
thus%
\begin{eqnarray*}
S_{1}&+&S_{2}+S_{3}
:=\frac{1}{2\pi \sqrt{\left\vert \sigma \right\vert }}\overset{s}{\underset%
{j=1}{\sum }}\mu _{j}\left( R_{n}\right) _{j}\log n \\
&+&\frac{1}{2\pi \sqrt{%
\left\vert \sigma \right\vert }}\overset{s}{\underset{j=1}{\sum }}\mu
_{j}\left( R_{n}\right) _{j}\log \widehat{c}_{j}+\underset{j=1}{\overset{s}{%
\sum }}O\left( n^{-1/2}\right) \left( R_{n}\right) _{j}\leq 1.
\end{eqnarray*}%
Since $\left( R_{n}\right) _{j}\rightarrow 0$, it follows that $%
S_{2}+S_{3}=o(1)$. So we have the upper estimation%
\begin{equation}
\gamma \left( n\right) \leq \frac{2\pi \sqrt{\left\vert \sigma \right\vert }%
}{\log n}+o\left( \frac{1}{\log n}\right) .  \label{twodimupperbound1}
\end{equation}%
From (\ref{twodimlowebound6}) and (\ref{twodimupperbound1}) we get%
\begin{equation}
\gamma \left( n\right) =\frac{2\pi \sqrt{\left\vert \sigma \right\vert }}{%
\log n}+o\left( \frac{1}{\log n}\right) .  \label{twodimeasyest}
\end{equation}%
Unfortunately, the estimation (\ref{twodimeasyest}) is not good enough for
our purposes (but observe that we have not really used (\ref%
{limitpropmoredimension}) yet). Now, (\ref{twodimupperbound1}) yields $\left(
R_{n}\right) _{j}=O\left( \frac{1}{\log n}\right) $ for all $1\leq j\leq s$.
Hence, with the previous notation, $S_{2}=O\left( \frac{1}{\log n}\right) $.
Obviously $S_{3}=O\left( \frac{1}{\log n}\right) $. Thus we arrived at%
\begin{equation}
\gamma (n)\leq \frac{2\pi \sqrt{\left\vert \sigma \right\vert }}{\log n}%
+O\left( \frac{1}{\log ^{2}n}\right) .  \label{twodimsharpupperbound}
\end{equation}%
This estimation will be sharp enough.

Now, we have to improve our lower estimation. From (\ref{twodimeasyest}) and
(\ref{twodimlowebound45}) it follows that 
\begin{equation*}
\gamma \left( n-k\right) \left[ \frac{1}{2\pi \sqrt{\left\vert \sigma
\right\vert }}\log k+O\left( 1\right) \right] +C\log \frac{n}{k}+O\left( k^{-%
\frac{1}{2}}\right) \geq 1,
\end{equation*}%
thus%
\begin{equation*}
\gamma \left( n-k\right) \log \left( n-k\right) \geq \left( 2\pi \sqrt{%
\left\vert \sigma \right\vert }-C2\pi \sqrt{\left\vert \sigma \right\vert }%
\log \frac{n}{k}+O\left( k^{-\frac{1}{2}}\right) \right) \frac{\log \left(
n-k\right) }{\log k+O(1)}.
\end{equation*}%
Now, similarly to the case of \cite{Dv-Er}, it follows that%
\begin{equation}
\gamma \left( n\right) =\frac{2\pi \sqrt{\left\vert \sigma \right\vert }}{%
\log n}+O\left( \frac{\log \log n}{\log ^{2}n}\right) .  \label{twodimend1}
\end{equation}%
Now, an elementary calculation completes the proof.
\end{proof}

As in the high dimensional case, the initial distribution does not influence
the asymptotic behavior. More precisely

\begin{proposition}
\label{twodimednprop}The assertion of Theorem \ref{twodimEdn} remains true
when the distribution of $\varepsilon _{0}$ is arbitrary.
\end{proposition}

\begin{proof}
The proof is very similar to the one of Proposition \ref{highdimEdnprop}. We
know that 
\begin{equation*}
\gamma \left( n\right) =\frac{2\pi \sqrt{\left\vert \sigma \right\vert }}{%
\log n}+O\left( \frac{\log \log n}{\log ^{2}n}\right) .
\end{equation*}%
With the notation $\gamma ^{e_{j}}\left( n\right) =\frac{2\pi \sqrt{%
\left\vert \sigma \right\vert }}{\log n}+h^{j}\left( n\right) $ our aim is
to prove $h^{j}\left( n\right) =O\left( \frac{\log \log n}{\log ^{2}n}%
\right) $. The analogue of (\ref{highprop1}) is%
\begin{eqnarray*}
&&\frac{2\pi \sqrt{\left\vert \sigma \right\vert }}{\log n}+h^{j}\left(
n\right) \\ 
&=&\overset{s}{\underset{k=1}{\sum }}\left[ \left( \mu
_{k}+b_{k}^{j}(K)\right) \left( \frac{2\pi \sqrt{\left\vert \sigma
\right\vert }}{\log \left( n-K\right) }+h^{k}\left( n-K\right) \right) %
\right] +O\left( K\cdot \left( n-K\right) ^{-1}\right) ,
\end{eqnarray*}%
and the analogue of (\ref{highprop2}) is%
\begin{eqnarray*}
h^{j}(n) &=&\overset{s}{\underset{k=1}{\sum }}\mu _{k}h^{k}(n-K)+\overset{s}%
{\underset{k=1}{\sum }}b_{k}^{j}(K)h^{k}(n-K)+O\left( K\cdot \left(
n-K\right) ^{-1}\right) \\
&&+\left( \frac{2\pi \sqrt{\left\vert \sigma \right\vert }}{\log \left(
n-K\right) }-\frac{2\pi \sqrt{\left\vert \sigma \right\vert }}{\log n}\right)
\\
&=&:I+II+III+IV.
\end{eqnarray*}%
With the choice $K\left( n\right) =\left\lfloor \sqrt{n}\right\rfloor $
elementary calculations show that $I+II+III+IV\leq O\left( \frac{\log \log n%
}{\log ^{2}n}\right) $.
\end{proof}

Now let us see the estimation of the variance.

\begin{theorem}
\label{twodimvdn}If (\ref{harmadikmomentum}) exists, then we have with
arbitrary $\nu $\ distribution of $\varepsilon _{0}$ 
\begin{equation*}
V_{2}(n)=O\left( \frac{n^{2}\log \log n}{\log ^{3}n}\right) .
\end{equation*}%
Moreover, the great order is uniform in $\nu $.
\end{theorem}

\begin{proof}
First, suppose $\varepsilon _{0}\sim \mu $. The beginning of the proof of
this case is the same as in Theorem \ref{highdimvdn}. The difference is that
when we change $K$ to $K\left( n\right) $, we can write $O\left( \frac{\log
\log K\left( n\right) }{\log K\left( n\right) }\right) $ instead of $\delta $
in the sense of Proposition \ref{twodimednprop}. From now, just like in the
proof of Theorem \ref{highdimvdn}, it is not difficult to deduce that%
\begin{equation*}
O\left( \frac{n}{\log n}\right) \left[ \frac{\log
\log K\left( n\right) }{\log K\left( n\right) }O\left( \frac{n}{\log n}%
\right) +O\left( \frac{n\log \log n}{\log ^{2}n}\right) \right] +K\left(
n\right) O\left( \frac{n}{\log n}\right) 
\end{equation*}%
is an upper bound for $V_{2}\left( n\right)$. Taking $K\left( n\right) =\left\lfloor \frac{n}{\log ^{2}n}\right\rfloor $
proves the statement. For the case of arbitrary initial distribution, one
can repeat the proof of Proposition \ref{highdimvdnprop}.
\end{proof}

\begin{corollary}
For a RWwIS in $d=2$ dimension weak law of large numbers holds.
\end{corollary}

\begin{proof}
Since $O\left( \frac{n^{2}\log \log n}{\log ^{3}n}\right) <O\left( \frac{%
n^{2}}{\log ^{2}n}\right) $, Chebyshev's inequality applies.
\end{proof}

The proof of the strong law of large numbers is quite complicated, so we
treat it in a different Section.

\section{Law of large numbers in the plane}
\label{Sectstronglaw}

This Section is dedicated to the strong law in $d=2$.

\begin{theorem}
For any RWwIS in $d=2$, for which (\ref{negyedikmomentum}) exists, strong
law of large numbers holds, namely%
\begin{equation*}
P\left( \underset{n\rightarrow \infty }{\lim }\frac{L_{2}(n)}{E_{2}(n)}%
=1\right) =1.
\end{equation*}
\end{theorem}

Almost the whole proof in \cite{Dv-Er} can be easily generalized to our case
with the observation that since our estimations for $E_{2}(n)$ and $V_{2}(n)$%
\ are uniform in the initial distribution, the computations, used in \cite%
{Dv-Er},\ can be repeated. That is why we write here the only non-trivial
part (i.e. formulae corresponding to (5.13) and (5.15) in \cite{Dv-Er}) of
the generalization. In fact, there is apparently a gap in the argument in 
\cite{Dv-Er}, as it was already remarked in \cite{J-P}. What we represent
here is a simplified version of a proof in \cite{Pene}. For the other parts
of the proof the reader is referred to \cite{Dv-Er}.

\begin{proof}
Denote%
\begin{equation*}
K=\left\lfloor \log \log n\right\rfloor ,
\end{equation*}%
and let $M_{ij}$ $(1\leq i,j\leq K)$ be the number of lattice points which
are common in path parts $M_{i}$ and $M_{j}$, where$\ M_{i}$ denotes the set
of points which are visited between $\left\lfloor (i-1)n/K\right\rfloor +1$
and $\left\lfloor in/K\right\rfloor $ ($1\leq i\leq K$). First, we would
like to prove the formula corresponding to \cite{Dv-Er} (5.13):%
\begin{equation}
\underset{i<j}{\sup }E\left( M_{ij}\right) =O\left( \frac{n\log \log n}{%
K\log ^{2}n}\right) .  \label{apr21egy}
\end{equation}%
If it is done, then for every $\vartheta $ with $0<\vartheta <1$ we will have%
\begin{equation}
\underset{i<j}{\sup }P\left( M_{ij}>\frac{n\log \log n}{K\log ^{1+\vartheta
}n}\right) =O\left( \frac{1}{\log ^{1-\vartheta }n}\right) .
\label{apr21ketto}
\end{equation}%
Let $C_{ij}$ denote the event whose probability is estimated in (\ref%
{apr21ketto}). As (\ref{apr21egy}) yields 
\begin{equation*}
\underset{j}{\sup }E\left( M_{1j}\right) =O\left( \frac{n\log \log n}{K\log
^{2}n}\right)
\end{equation*}%
for arbitrary $\nu $ initial distribution of internal states, and under the
condition $C_{ij}$ the probability of $C_{i^{\prime }j^{\prime }}$ with $%
1\leq i<j<i^{\prime }<j^{\prime }\leq K$ is only affected via the
distribution of $\varepsilon _{i^{\prime }}$, we conclude%
\begin{equation}
\underset{1\leq i<j<i^{\prime }<j^{\prime }\leq K}{\sup }P(C_{i,j}\cap
C_{i^{\prime },j^{\prime }})=O\left( \frac{1}{\log ^{2-2\vartheta }n}\right)
.  \label{apr21kettoot}
\end{equation}%
If we were able to prove%
\begin{equation}
\underset{i,j,i^{\prime },j^{\prime }}{\sup }E\left( M_{ij}M_{i^{\prime
}j^{\prime }}\right) =O\left( \frac{n^{2}\log ^{2}\log n}{K^{2}\log ^{4}n}%
\right) ,  \label{apr21harom}
\end{equation}%
where the supremum is taken over indices for which $\#\left\{ i,j,i^{\prime
},j^{\prime }\right\} =4$ and either $1\leq i<i^{\prime }<j^{\prime }<j\leq
K $ or $1\leq i<i^{\prime }<j<j^{\prime }\leq K$ holds, then using 
\begin{equation*}
P\left( M_{ij}>\frac{n\log \log n}{K\log ^{1+\vartheta }n},M_{i^{\prime
}j^{\prime }}>\frac{n\log \log n}{K\log ^{1+\vartheta }n}\right) <P\left(
M_{ij}M_{i^{\prime }j^{\prime }}>\frac{n^{2}\log ^{2}\log n}{K^{2}\log
^{2+2\vartheta }n}\right)
\end{equation*}%
and (\ref{apr21kettoot}) we could infer that the probability that two events 
$C_{ij}$ and $C_{i^{\prime }j^{\prime }}$ with $\#\left\{ i,j,i^{\prime
},j^{\prime }\right\} =4$ occur is%
\begin{equation}
O\left( \frac{K^{4}}{\log ^{2-2\vartheta }n}\right) ,  \label{apr21negy}
\end{equation}%
which is the formula corresponding to \cite{Dv-Er}\ (5.15).

So our aim is to prove (\ref{apr21egy}) and (\ref{apr21harom}). The idea of 
\cite{Pene} is that in order to prove (\ref{apr21egy}) and (\ref{apr21harom}%
) it is useful to cut down the points of $M_{i}$ which are visited in the
extreme $\left\lceil n/\log ^{2}n\right\rceil $ steps. The number of these
points can be roughly estimated, while the others are visited in steps quite
far from each other and this will be enough for us. However, the precise
arguments need some awkward computations.

\emph{Proof of\ (\ref{apr21egy})} We introduce the notations%
\begin{equation*}
\alpha _{a,b}=\left\{ \forall t=a,...,b-1:\eta _{t}\neq \eta _{b}\right\} 
\text{ and }\beta _{a,b}=\left\{ \forall t=a+1,...,b:\eta _{a}\neq \eta
_{t}\right\}
\end{equation*}%
which will be useful in the sequel. Following \cite{Pene}, we define%
\begin{equation*}
n_{(i,-)}=\left\lfloor (i-1)n/K\right\rfloor +\left\lceil n/\log
^{2}n\right\rceil \text{ and }n_{(i,+)}=\left\lfloor in/K\right\rfloor
-\left\lceil n/\log ^{2}n\right\rceil .
\end{equation*}%
A point, which is common in the paths $\eta _{n_{(i,-)}},...,\eta
_{n_{(i,+)}}$ and $\eta _{n_{(j,-)}},...,\eta _{n_{(j,+)}}$ and not visited
in the extreme $\left\lceil n/\log ^{2}n\right\rceil $ steps of $M_{i}$ and $%
M_{j}$, has a pair of indices $(k,l)$, $k\in n_{(i,-)},...,n_{(i,+)}$, $l\in
n_{(j,-)},...,n_{(j,+)}$, such that it is visited at steps $k$ and $l$, and
it is not visited during steps $\left\lfloor (i-1)n/K\right\rfloor
+1,...,k-1 $, and steps $l+1,...,\left\lfloor jn/K\right\rfloor $. So we have%
\begin{eqnarray*}
E\left( M_{ij}\right) &\leq &\frac{3n}{\log ^{2}n}+\underset{k=n_{(i,-)}}{%
\overset{n_{(i,+)}}{\sum }}\underset{l=n_{(j,-)}}{\overset{n_{(j,+)}}{\sum }}%
P\left( \alpha _{\left\lfloor (i-1)n/K\right\rfloor +1,k}\cap \left\{ \eta
_{k}=\eta _{l}\right\} \cap \beta _{l,\left\lfloor jn/K\right\rfloor }\right)
\\
&\leq &\frac{3n}{\log ^{2}n} \\
&&+C_{1}\underset{k=n_{(i,-)}}{\overset{n_{(i,+)}}{%
\sum }}\underset{l=n_{(j,-)}}{\overset{n_{(j,+)}}{\sum }}\frac{1}{\log
\left( k-\left\lfloor (i-1)n/K\right\rfloor \right) }\frac{1}{l-k}\frac{1}{%
\log \left( \left\lfloor jn/K\right\rfloor -l\right) } \\
&\leq &\frac{3n}{\log ^{2}n}+C_{2}\frac{n}{K}\frac{\log n-\log \left( n/\log
^{2}n\right) }{\log ^{2}n}=O\left( \frac{n\log \log n}{K\log ^{2}n}\right) .
\end{eqnarray*}%
Note that we have used our estimations for the probability of avoiding the
origin in some steps, visiting a new point, and returning to the origin, and
these estimations are uniform in the initial distribution (with an
appropriate $C_{1}$). Because the events whose intersection's probability is
estimated above are dependent only via the internal states, it is obvious
that the great order is uniform in $i$ and $j$. So we arrived at (\ref%
{apr21egy}).

\emph{Proof of (\ref{apr21harom})} Let us prove%
\begin{equation*}
\underset{1\leq i<i^{\prime }<j^{\prime }<j\leq K}{\sup }E\left(
M_{ij}M_{i^{\prime }j^{\prime }}\right) =O\left( \frac{n^{2}\left( \log \log
n\right) ^{2}}{K^{2}\log ^{4}n}\right) .
\end{equation*}%
Let us introduce the notation $\mathcal{L}$ for the set of $\left(
k,k^{\prime },l^{\prime },l\right) $ such that%
\begin{eqnarray*}
&&n_{(i,-)}\leq k\leq n_{(i,+)},\text{ }n_{(i^{\prime },-)}\leq k^{\prime
}\leq n_{(i^{\prime },+)},  \\
&&n_{(j^{\prime },-)}\leq l^{\prime }\leq
n_{(j^{\prime },+)},\text{ }n_{(j,-)}\leq l\leq n_{(j,+)}.
\end{eqnarray*}%
As it was mentioned before, we estimate the number of pair of points one of
which is visited in either extreme $\left\lceil n/\log ^{2}n\right\rceil $
steps of $M_{i}$, $M_{j}$, $M_{i'}$ or $M_{j'}$  in a very obvious manner. The other pairs of
lattice points ($x$ and $y$, say) have a $\left( k,k^{\prime },l^{\prime
},l\right) $ element of $\mathcal{L}$, such that $x$ is visited at step $k$
but not visited during $\left\lfloor (i-1)n/K\right\rfloor +1,...,k-1$, and
it is visited again at step $l$  but not visited during $l+1,...,\left\lfloor jn/K\right\rfloor $; while $y$ is visited at
step $k^{\prime }$ but not visited during $k^{\prime}+1,...,\left\lfloor i^{\prime }n/K\right\rfloor $, and it is visited again at step $l^{\prime }$ but not visited during 
$\left\lfloor (j^{\prime}-1)n/K\right\rfloor +1,...,l^{\prime }-1$. So we
have%
\begin{equation}
E\left( M_{ij}M_{i^{\prime }j^{\prime }} \right) =O\left( \frac{n^{2}\log
\log n}{K\log ^{4}n}\right) +\underset{M\in \mathbb{Z} ^{2}}{\sum }%
\underset{\left( k,k^{\prime },l^{\prime },l\right) \in \mathcal{L}}{\sum }%
P\left( \mathcal{A}\right) ,  \label{apr21negyot}
\end{equation}%
where%
\begin{eqnarray*}
\mathcal{A} &\mathcal{=}&\alpha _{\left\lfloor (i-1)n/K\right\rfloor
+1,k}\cap \left\{ \eta _{k^{\prime }}-\eta _{k}=M\right\} \cap \beta
_{k^{\prime },\left\lfloor i^{\prime }n/K\right\rfloor }\cap \left\{ \eta
_{l^{\prime }}-\eta _{k^{\prime }}=(0,0)\right\} \\
&&\cap \alpha _{\left\lfloor (j^{\prime }-1)n/K\right\rfloor +1,l^{\prime
}}\cap \left\{ \eta _{l}-\eta _{l^{\prime }}=-M\right\} \cap \beta
_{l,\left\lfloor jn/K\right\rfloor },
\end{eqnarray*}%
Denote the seven events, whose intersection is $\mathcal{A}$, by $\mathcal{A}%
_{1},...,\mathcal{A}_{7}$. Observe that for every $2\leq m\leq 7$ the
probability of $\mathcal{A}_{m}$ under the condition $\mathcal{A}_{1}\cap
...\cap \mathcal{A}_{m-1}$ is just the probability of $\mathcal{A}_{m}$ with
an appropriate initial distribution of $\varepsilon $. As we have uniform
estimations in the initial distribution, we will be able to use them.

In the first step, let us estimate the part of the sum in (\ref{apr21negyot}) corresponding
to $M \in [ -n, n ]^2 $.
Proposition \ref{Penelemma} yields the existence of $a>0$ (which depends
only on the RWwIS), such that%
\begin{eqnarray*}
P\left( \eta _{k}-\eta _{0}=M\right) &<&C_{3}\exp \left( -\frac{a}{2k}%
M^{T}M\right) \left[ \frac{1}{k}+\frac{1}{k^{3/2}}\right] +\frac{C_{3}}{k^{2}%
} \\
&<&C_{4}\left( \frac{1}{k}\exp \left( -\frac{a}{2k}M^{T}M\right) +\frac{1}{%
k^{2}}\right) .
\end{eqnarray*}%
So the formula%
\begin{equation}
C_{5}\frac{1}{\log n}\left( \frac{1}{%
k^{\prime }-k}\exp \left( -\frac{a}{2(k^{\prime }-k)}M^{T}M\right) +\frac{1}{%
(k^{\prime }-k)^{2}}\right) ,  \label{apr21ot}
\end{equation}%
is an upper bound for $P(\mathcal{A}_{1}\cap \mathcal{A}_{2})$, and the formula%
\begin{equation}
C_{5}\frac{1}{\log n}\left( \frac{1}{l-l^{\prime }}\exp \left( -\frac{a%
}{2(l-l^{\prime })}M^{T}M\right) +\frac{1}{(l-l^{\prime })^{2}}\right) .
\label{apr21hat}
\end{equation}%
is an upper bound for $P(\mathcal{A}_{6}\cap \mathcal{A}_{7}|\mathcal{A}_{1}\cap ...\cap \mathcal{A}_{5})$. \\
Consider the following factorization%
\begin{equation}
P(\mathcal{A})=P(\mathcal{A}_{1}\cap \mathcal{A}_{2})P(\mathcal{A}_{3}\cap 
\mathcal{A}_{4}\cap \mathcal{A}_{5}|\mathcal{A}_{1}\cap \mathcal{A}_{2})P(%
\mathcal{A}_{6}\cap \mathcal{A}_{7}|\mathcal{A}_{1}\cap ...\cap \mathcal{A}%
_{5}),  \label{apr21het}
\end{equation}%
and observe that%
\begin{equation}
\underset{k^{\prime },l^{\prime }}{\sum }P(\mathcal{A}_{3}\cap \mathcal{A}%
_{4}\cap \mathcal{A}_{5}|\mathcal{A}_{1}\cap \mathcal{A}_{2})<C_{6}E(M_{i^{%
\prime }j^{\prime }})=O\left( \frac{n\log \log n}{K\log ^{2}n}\right) .
\label{a2010} \end{equation} %
So we have to take the product of the expressions in (\ref{apr21ot}), (\ref%
{apr21hat}) and $P(\mathcal{A}_{3}\cap \mathcal{A}_{4}\cap \mathcal{A}_{5}|%
\mathcal{A}_{1}\cap \mathcal{A}_{2})$ and sum them up in all of the four
indices to estimate (\ref{apr21negyot}). First, let us consider the product
of the first terms in (\ref{apr21ot}) and (\ref{apr21hat}). We have to
estimate%
\begin{equation*}
{\sum }{\sum }\exp \left( -\frac{a}{2}\left( 
\frac{1}{k^{\prime }-k}+\frac{1}{l-l^{\prime }}\right) M^{T}M\right) \frac{P(%
\mathcal{A}_{3}\cap \mathcal{A}_{4}\cap \mathcal{A}_{5}|\mathcal{A}_{1}\cap 
\mathcal{A}_{2})}{(k^{\prime }-k)(l-l^{\prime })\log ^{2}n},
\end{equation*}%
where the to sums are taken over $M\in [-n, n]^2$ and $\left( 
k,k^{\prime},l^{\prime },l\right) \in \mathcal{L}$, respectively.
Using the fact%
\begin{equation}
\underset{d>a}{\sup }\frac{1}{d}\underset{M\in 
\mathbb{Z}
^{2}}{\sum }\exp \left( -\frac{a}{2d}M^{T}M\right) <+\infty  \label{apr24egy}
\end{equation}%
it suffices to estimate%
\begin{eqnarray*}
&&\underset{\left( k,k^{\prime },l^{\prime },l\right) \in \mathcal{L}}{\sum }%
\frac{(k^{\prime }-k)(l-l^{\prime })}{(l-l^{\prime })+(k^{\prime }-k)}\frac{%
P(\mathcal{A}_{3}\cap \mathcal{A}_{4}\cap \mathcal{A}_{5}|\mathcal{A}%
_{1}\cap \mathcal{A}_{2})}{(k^{\prime }-k)(l-l^{\prime })\log ^{2}n} \\
&\leq &\frac{1}{\log ^{2}n}\underset{\left( k,k^{\prime },l^{\prime
},l\right) \in \mathcal{L}}{\sum }\frac{P(\mathcal{A}_{3}\cap \mathcal{A}%
_{4}\cap \mathcal{A}_{5}|\mathcal{A}_{1}\cap \mathcal{A}_{2})}{%
(l-n_{(j^{\prime },+)})+(n_{(i^{\prime },-)}-k)}.
\end{eqnarray*}%
Using (\ref{a2010}), it remains to estimate%
\begin{equation*}
\frac{1}{\log ^{2}n}E\left( M_{ij}\right) \underset{k,l}{\sum }\frac{1}{%
(l-n_{(j,-)})+(n_{(i,+)}-k)+2n/\log ^{2}n-1}
\end{equation*}%
and it is just%
\begin{equation*}
O\left( \frac{n^{2}\left( \log \log n\right) ^{2}}{K^{2}\log ^{4}n}\right)
\end{equation*}%
uniformly in $i$ and $j$, by an elementary computation.

Now, let us consider the product of the first term in (\ref{apr21ot}) and
the second term in (\ref{apr21hat}) (the product of the second term in (\ref%
{apr21ot}) and the first term in (\ref{apr21hat}) can be estimated
equivalently). In this case the easier estimation%
\begin{equation}
P(\mathcal{A}_{3}\cap \mathcal{A}_{4}\cap \mathcal{A}_{5}|\mathcal{A}%
_{1}\cap \mathcal{A}_{2})<C_{8}\frac{1}{l^{\prime }-k^{\prime }}
\label{apr24ketto}
\end{equation}%
will be enough. Thus our aim is to estimate%
\begin{equation*}
\frac{1}{\log ^{2}n}\underset{\left( k,k^{\prime },l^{\prime },l\right) \in 
\mathcal{L}}{\sum }\underset{M\in [-n,n]^{2}}{\sum }\exp \left(
-\frac{a}{2\left( k^{\prime }-k\right) }M^{T}M\right) \frac{1}{(k^{\prime
}-k)(l-l^{\prime })^{2}(l^{\prime }-k^{\prime })}.
\end{equation*}%
As above, we use (\ref{apr24egy}) to handle the exponential terms. So the
following estimation is enough for our purposes%
\begin{eqnarray*}
\frac{1}{\log ^{2}n}\underset{\left( k,k^{\prime },l^{\prime },l\right) \in 
\mathcal{L}}{\sum }\frac{1}{(l-l^{\prime })^{2}(l^{\prime }-k^{\prime })}
&\leq &\frac{1}{\log ^{2}n}\frac{n^{4}}{K^{4}}\frac{\log ^{4}n}{n^{2}}\frac{%
\log ^{2}n}{n} \\
&=&O\left( \frac{n^{2}\left( \log \log n\right) ^{2}}{K^{2}\log ^{4}n}%
\right) .
\end{eqnarray*}

Our last task is to estimate the product of the second term in (\ref{apr21ot}%
) and (\ref{apr21hat}). The previous estimation (\ref{apr24ketto}) and%
\begin{eqnarray*}
&&\frac{1}{\log ^{2}n}\underset{\left( k,k^{\prime },l^{\prime },l\right) \in 
\mathcal{L}}{\sum }\underset{M\in \left[ -n,n\right] ^{2}}{\sum }\frac{1}{%
(k^{\prime }-k)^{2}(l-l^{\prime })^{2}(l^{\prime }-k^{\prime })} \\ 
&\leq &n^{6}%
\frac{1}{\log ^{2}n}\frac{\log ^{4}n}{n^{2}}\frac{\log ^{4}n}{n^{2}}\frac{%
\log ^{2}n}{n}
=O\left( \frac{n^{2}\left( \log \log n\right) ^{2}}{K^{2}\log ^{4}n}\right)
\end{eqnarray*}%
yield the required estimation.

In the second step, we estimate the part of the sum in (\ref{apr21negyot}) 
corresponding to $M \in \mathbb{Z} ^2 \setminus [-n,n]^2 $. 
Corollary \ref{cor2010} implies that
\begin{equation*}
\sup_{(k,k',l,l') \in \mathcal{L} }
\underset{M \in \mathbb{Z}^2 \setminus [-n,n]^2}{\sum }P(
\mathcal{A}_{6}|\mathcal{A}_{1}\cap ...\cap \mathcal{A}_5)
<  \frac{1}{n^{1/4- \varepsilon}}
\end{equation*}
for small $\varepsilon >0 $. Thus
\begin{eqnarray*}
&&\sup _{k', l'} \underset{k,l}{\sum }
\sup _{M \in \mathbb{Z}^2 \setminus [-n,n]^2} P(\mathcal{A}_{1}\cap \mathcal{A}_{2})
\underset{M \in \mathbb{Z}^2 \setminus [-n,n]^2}{\sum }P(
\mathcal{A}_{6}\cap \mathcal{A}_{7}|\mathcal{A}_{1}\cap ...\cap \mathcal{A}%
_{5}) \\
&<& C_7 n^2 \frac{1}{\log n} \frac{\log ^2 n}{n}
\frac{1}{\log n} \frac{1}{n^{1/4- \varepsilon}}.
\end{eqnarray*}
The above estimation together with (\ref{a2010}) yield the required error term. 

A modified version of the proof presented above can be repeated for indices $%
1\leq i<i^{\prime }<j<j^{\prime }\leq K$. So we have finished the proof of
formula (\ref{apr21harom}).
\end{proof}

\section{Visited points in one dimension}
\label{Sectonedim}

Investigating the one dimensional case is not as important as the higher
dimensions, as Lorentz processes used to be examined mainly in higher
dimensions. However, one dimension is also interesting, as we will see some
new features. We need some different means from the previous ones to prove
asymptotics for $E_{1}\left( n\right) $, namely Tauberian arguments. Let us
see the details.

\begin{proposition}
\label{gamma1dim}For a one dimensional RWwIS with $\varepsilon _{0}\sim \mu $%
\ we have%
\begin{equation*}
\gamma _{1}(n)\sim \sqrt{\frac{2\left\vert \sigma \right\vert }{\pi }}\ast
n^{-1/2}
\end{equation*}
\end{proposition}

\begin{proof}
Just like in the higher dimensional cases we consider the renewal equation
for the reversed walk%
\begin{equation*}
\overset{n}{\underset{k=0}{\sum }}U_{k}\cdot R_{n-k}=\underline{1}.
\end{equation*}%
Now, from row $i$ we obtain%
\begin{equation}
\overset{s}{\underset{j=1}{\sum }}\overset{n}{\underset{k=0}{\sum }}\left(
U_{k}\right) _{i,j}x^{k}\left( R_{n-k}\right) _{j}x^{n-k}=x^{n}.
\label{feb16egy}
\end{equation}%
Let us introduce the notations%
\begin{eqnarray*}
\overset{\infty }{\underset{k=0}{\sum }}\left( U_{k}\right) _{i,j}x^{k}
&=&\alpha _{ij}\left( x\right) \\
\overset{\infty }{\underset{k=0}{\sum }}\left( R_{k}\right) _{j}x^{k}
&=&\beta _{j}\left( x\right) \\
\overset{\infty }{\underset{k=0}{\sum }}x^{k} &=&\omega \left( x\right) .
\end{eqnarray*}%
Obviously, these power series are convergent for $0\leq x<1$. In these
terms, (\ref{feb16egy}) means%
\begin{equation}
\overset{s}{\underset{j=1}{\sum }}\alpha _{ij}\beta _{j}=\omega .
\label{feb16ketto}
\end{equation}%
In order to obtain the order of the coefficients of $\gamma _{1}\left(
n\right) =\overset{s}{\underset{j=1}{\sum }}\mu _{j}\left( R_{n}\right)
_{j} $ we use a Tauberian theorem which may be found in \cite{Feller}
(Theorem 5 of XIII.5). According to this we have%
\begin{equation}
\omega \left( x\right) \sim \frac{1}{1-x},\qquad x\rightarrow 1-.
\label{feb16harom}
\end{equation}%
For the coefficients of $\alpha _{ij}$%
\begin{equation*}
\overset{n}{\underset{k=0}{\sum }}\left( U_{k}\right) _{i,j}\sim 2\frac{1}{%
\sqrt{2\pi \left\vert \sigma \right\vert }}\mu _{j}n^{1/2}.
\end{equation*}%
So, using the Tauberian theorem, we infer%
\begin{equation}
\alpha _{ij}\left( x\right) \sim 2\frac{1}{\sqrt{2\pi \left\vert \sigma
\right\vert }}\mu _{j}\Gamma \left( \frac{3}{2}\right) \ast \frac{1}{\left(
1-x\right) ^{1/2}},\qquad x\rightarrow 1-.  \label{feb16negy}
\end{equation}%
From (\ref{feb16ketto}) we obtain%
\begin{equation}
\overset{s}{\underset{j=1}{\sum }}\frac{\alpha _{ij}}{\alpha _{ii}}\beta
_{j}=\frac{\omega }{\alpha _{ii}}.  \label{feb16ot}
\end{equation}%
Now, (\ref{feb16negy}) yields 
\begin{equation}
\frac{\alpha _{ij}\left( x\right) }{\alpha _{ii}\left( x\right) }\rightarrow 
\frac{\mu _{j}}{\mu _{i}},\qquad x\rightarrow 1-.  \label{feb16hat}
\end{equation}%
Whence 
\begin{equation*}
\overset{s}{\underset{j=1}{\sum }}\mu _{j}\beta _{j}\left( x\right) \sim 
\frac{\sqrt{2\pi \left\vert \sigma \right\vert }}{2\Gamma \left( \frac{3}{2}%
\right) }\ast \frac{1}{\left( 1-x\right) ^{1/2}},\qquad x\rightarrow 1-.
\end{equation*}%
Since $\overset{s}{\underset{j=1}{\sum }}\mu _{j}\left( R_{k}\right) _{j}$
is monotonic in $k$, using the mentioned Tauberian theorem we conclude%
\begin{equation*}
\gamma _{1}\left( n\right) =\overset{s}{\underset{j=1}{\sum }}\mu
_{j}\left( R_{k}\right) _{j}\sim \frac{\sqrt{2\pi \left\vert \sigma
\right\vert }}{2\Gamma \left( \frac{3}{2}\right) \Gamma \left( \frac{1}{2}%
\right) }n^{-1/2}=\sqrt{\frac{2\left\vert \sigma \right\vert }{\pi }}n^{-1/2}
\end{equation*}
\end{proof}

\begin{proposition}
With arbitrary distribution of $\varepsilon _{0}$ the following holds%
\begin{equation*}
E_{1}\left( n\right) \sim \sqrt{\frac{8\left\vert \sigma \right\vert }{\pi }}%
n^{1/2}.
\end{equation*}
\end{proposition}

\begin{proof}
From Proposition \ref{gamma1dim} the assertion immediately follows in the
case of $\varepsilon _{0}\sim \mu $. However, the case of arbitrary initial distribution requires a little care. Analogously to (\ref{highprop1}), we
have%
\begin{equation}
\gamma ^{e_{j}}(n)=\overset{s}{\underset{k=1}{\sum }}\mu _{k}\gamma
^{e_{k}}(n-K)+\overset{s}{\underset{k=1}{\sum }}b_{k}^{j}(K)\gamma
^{e_{k}}(n-K)+p(K,n). \label{last}
\end{equation}%
But now, the rough estimation of $p(K,n)$ used in higher dimensions is not enough,
as the local limit theorem provides a term of order $n^{-1/2}$ and our aim is to prove $o \left( n^{-1/2} \right)$.
Nevertheless, because of the definition of $p(K,n)$, we have to estimate the probability of the first return to some place after $m$ steps. In particular, if we proved that this probability is $O(m^{-3/2})$,
then taking $K=\left\lfloor \sqrt{n}\right\rfloor $ and multiplying (\ref{last}) by $\sqrt{n}$ we
would find that the right hand side converges to $\sqrt{\frac{2\left\vert \sigma
\right\vert }{\pi }}$ as $n\rightarrow \infty $.
So, in order to finish our proof, we need the following lemma.

\end{proof}

\begin{lemma}
\label{firstreturnorder}For a one dimensional RWwIS fulfilling our basic assumptions
with arbitrary $\nu $\ distribution of $\varepsilon _{0}$%
\begin{equation}
f^{\nu }(n)=O\left( n^{-3/2}\right),  \label{firstretd1}
\end{equation}%
where $f^{\nu }(n)$ denotes the probability of the event that the random walker starting from the origin with $\varepsilon _{0} \sim \nu$
returns to the origin at time $n$ for the first time.
\end{lemma}
\begin{proof}
First of all, observe that proving the statement for $\nu =\mu $ would be
enough as since our basic assumption (i) all component of $\mu $ are
positive.
In the proof,  we generalize an argument in \cite{Bender}. Define%
\begin{equation*}
Q^{n}(x,i,y,j)=P\left( \xi _{n}=(y,j),\eta _{k}\neq 0,\forall 1\leq
k<n|\xi _{0}=(x,i)\right) .
\end{equation*}%

Let $n=3m$ and $1\leq i\leq m$. The cases $n=3m\pm 1$ can be treated the
same way.%
\begin{eqnarray*}
&&f^{e_{i}}(n) =\underset{l=1}{\overset{s}{\sum }}Q^{n}(0,i,0,l)=%
\underset{y,z\neq 0}{\sum }\underset{j,k,l=1}{\overset{s}{\sum }}%
Q^{m}(0,i,y,j)Q_{d}^{m}(y,j,z,k)Q^{m}(z,k,0,l) \\
&\leq &\underset{y,z,j,k}{\sup }Q^{m}(y,j,z,k)P(\eta _{k}\neq 0,\forall
1\leq k<m|\xi _{0}=(0,i))\underset{z\neq 0}{\sum }\underset{k,l=1}{\overset{s%
}{\sum }}Q^{m}(z,k,0,l)
\end{eqnarray*}%
From the local limit theorem it follows that
\begin{equation*}
\underset{y,z,j,k}{\sup }Q^{m}(y,j,z,k)=O(m^{-1/2}).
\end{equation*}%
Proposition \ref{gamma1dim} yields $P(\eta _{k}\neq 0,\forall 1\leq k<m|\xi
_{0}=(0,i))=O(m^{-1/2})$. So, it suffices to prove%
\begin{equation}
\underset{z\neq 0}{\sum }\underset{k,l=1}{\overset{s}{\sum }}%
Q^{m}(z,k,0,l)=O(m^{-1/2}).  \label{firstreturnclaim1}
\end{equation}%
In order to prove (\ref{firstreturnclaim1}) we
use the reversed walk, again. (\ref{reversedwalk}) yields that for all $%
((0,i_{1}),(y_{1},i_{2}),(y_{1}+y_{2},i_{3}),...,(y_{1}+y_{2}+...+y_{m-1},i_{m}))
$ trajectories%
\begin{eqnarray*}
&& \mu _{i_{1}}p_{y_{1},i_{1},i_{2}}\mu _{i_{2}}p_{y_{2},i_{2},i_{3}}...\mu
_{i_{m-1}}p_{y_{m-1},i_{m-1},i_{m}} \\
&=& \mu _{i_{2}}q_{-y_{1},i_{2},i_{1}}\mu
_{i_{3}}q_{-y_{2},i_{3},i_{2}}...\mu _{i_{m}}q_{-y_{m-1},i_{m},i_{m-1}},
\end{eqnarray*}%
where the factors $\mu _{i_{2}},...,\mu _{i_{m-1}}$ drop out. Thus%
\begin{equation}
\underset{z\neq 0}{\sum }\underset{k,l=1}{\overset{s}{\sum }}%
Q^{m}(z,k,0,l)\leq \underset{1\leq i,j\leq s}{\max }\frac{\mu _{i}}{\mu
_{j}}\underset{z\neq 0}{\sum }\underset{k,l=1}{\overset{s}{\sum }}\widetilde{%
Q}^{m}(0,l,z,k),  \label{firstreturnproofend1}
\end{equation}%
where $\widetilde{Q}$ is the same object as $Q$ defined for the reversed
walk. The right hand side of (\ref{firstreturnproofend1}) can be bounded by
some constant times the probability of the event that the stationary
reversed walk does not return to the origin in the first $m$ steps, which is
$O(m^{-1/2})$. Thus we arrived at (\ref{firstreturnclaim1}).

\end{proof}

So, we have ascertained the asymptotic behavior of $E_{d}\left( n\right) $
in each dimension. While strong law of large numbers holds in $d \geq 2$, 
even the weak law of large numbers for one dimensional SSRW fails to hold, which is a consequence of, for instance, Theorem 1 in \cite{Chen}.

\section{Final remarks}
\label{Sectfinalremarks}

\begin{enumerate}
\item Our asymptotic investigations show that RWwIS behaves like the simple
symmetric random walk in an asymptotic sense. The main features are very
similar, only the involved constants differ. The results showing that the
asymptotic behavior is independent from the initial distribution on the
internal states (e.g. Proposition \ref{highdimEdnprop} and \ref%
{twodimednprop}) are intuitively trivial as after some steps $\varepsilon $
will be very close to $\mu $. Nevertheless, these assertions need formal
proofs as well, especially as they are used in the sequel. Of course, this
similarity to the simple symmetric random walk could change if the
generalization were carried further, for instance, if a countable set of
internal states was allowed. This model is not yet discussed, it must need
some more involved technics.

\item Our basic assumption (ii) is not essential. The above theorems could
be generalized to the case of dropping basic assumption (ii), as the limit
theorem in \cite{Kramli-Szasz} is proved for this case, as well. Only the
computations would become longer. The other three assumptions are essential.
\end{enumerate}

\begin{acknowledgements}
I am grateful to Domokos Sz\'{a}sz for his constant support and invaluable
comments before and during writing this paper. My research work was
partially supported by OTKA (Hungarian National Research Fund) grant
NK63066. I thank for the kind hospitality of Erwin Schr\"{o}dinger Institute
(Vienna), where a part of this work was done. I am also grateful to the referee
for his careful reading of this paper and for his pertinent remarks.
\end{acknowledgements}

\end{document}